\documentclass[10pt]{amsart}%
\usepackage[top=3cm, bottom=3cm, left=3cm, right=3cm]{geometry}
\usepackage{amssymb}
\usepackage{tikz}
\usepackage{color}
\usepackage{amsmath}
\usepackage{graphicx}
\usepackage{amsfonts}%
\providecommand{\U}[1]{\protect\rule{.1in}{.1in}}
\renewcommand{\phi}{\varphi}

\newtheorem{theorem}{Theorem}

\newtheorem{corollary}[theorem]{Corollary}

\newtheorem{definition}[theorem]{Definition}
\newtheorem{example}[theorem]{Example}

\newtheorem{lemma}[theorem]{Lemma}

\newtheorem{remark}[theorem]{Remark}

\email[Mar\'{\i}a J. Garrido-Atienza]{mgarrido@us.es}
\email[Bj{\"o}rn Schmalfu{\ss }]{bjoern.schmalfuss@uni-jena.de}
\email[Jose Valero]{jvalero@umh.es}

\begin{document}
\title[Setvalued dynamical systems for SPDEs driven by fBm]{Setvalued dynamical systems for stochastic evolution equations driven by
fractional noise}
\author{M.J. Garrido-Atienza}
\address[Mar\'{\i}a J. Garrido-Atienza]{Dpto. Ecuaciones Diferenciales y An\'alisis Num\'erico\\
Universidad de Sevilla, Apdo. de Correos 1160, 41080-Sevilla, Spain}
\author{B. Schmalfu{\ss }}
\address[Bj{\"o}rn Schmalfu{\ss }]{Institut f\"{u}r Stochastik\\
Friedrich Schiller Universit{\"a}t Jena, Ernst Abbe Platz 2, D-77043\\
Jena, Germany\\
 }
\author{J. Valero}
\address[J. Valero]{Centro de Investigaci\'on Operativa, Universidad Miguel
Hern\'andez, Avda. de la Universidad, s/n, 03202-Elche, Spain}

\begin{abstract}
We consider Hilbert-valued evolution equations driven by H\"{o}lder paths with
H\"{o}lder index greater than 1/2, which includes the case of fractional
noises with Hurst parameters in (1/2,1). The assumptions of the drift term
will not be enough to ensure the uniqueness of solutions. Nevertheless,
adopting a multivalued setting, we will prove that the set of all solutions
corresponding to the same initial condition generates a (multivalued)
nonautonomous dynamical system $\Phi$. Finally, to prove that $\Phi$ is
measurable (and hence a (multivalued) random dynamical system), we need to
construct a new metric dynamical system that models the noise with the
property that the set space is separable.

\end{abstract}
\maketitle

\section{Introduction}

In this article we aim at investigating the following type of evolution
equation
\begin{equation}
\label{P}du(t)=(Au(t)+F(u(t)))dt+G(u(t))d\omega(t),\quad u(0)=u_{0} \in V,
\end{equation}
in a Hilbert space $V$, where the driving input is given by a H\"older
continuous function with H\"older index greater than $1/2$, the operator $A$
generates an analytic semigroup, and the nonlinear mappings $F$ and $G$ will
be Lipschitz continuous.

The main example will be given by stochastic evolution equations driven by a
fractional Brownian motion $B^{H}$ with Hurst parameter $H\in(1/2,1)$.\newline

One of the main features of this paper is that the scarce regularity of $G$
will not be enough to ensure uniqueness of solutions of (\ref{P}), hence we
shall adopt the setvalued setting. When $\omega$ is a Brownian motion (which
corresponds to the case $B^{1/2}$), this kind of problems have been already
tackled. However, the corresponding solutions of these problems are defined
only almost surely, which contradicts the cocycle property. In other words,
the well-known Ito integral produces exceptional sets that depend on the
initial condition and it is not known therefore how to define a random
dynamical system if more than countably many exceptional sets occur. As a
result, it is not known in general if a stochastic evolution equation driven
by Brownian motion generates a random dynamical system, even in the univalued
case where uniqueness of solutions holds true. This general open problem has
been solved in very particular situations where the noise is additive or
linear multiplicative. In those cases, the technique consists in transforming
the stochastic equation into a random equation where the noise acts just as a
parameter, which in turns implies that deterministic tools can be used to deal
with the random equation. This method came out in the 90's of the last century
and since then has been exploited in many papers, as for instance,
\cite{CHSchV}, \cite{G}, \cite{GLR}, \cite{Gu} in the case of uniqueness of
solutions, and \cite{CGSchV10}, \cite{CLV} in the multivalued setting.\newline

Recently, many efforts have been done to go beyond the Brownian motion case
and to consider other types of noises that exhibit different properties, as is
the case of a fractional Brownian motion $B^{H}$ with Hurst parameter
$H\in(0,1)$. When $H=1/2$ this process reduces to the Brownian motion but, in
the rest of cases, $B^{H}$ is not a semimartingale and it is not a Markov
process. It would be impossible to mention here the huge amount of papers that
during the last decade came out with different investigations related to
equations driven by $B^{H}$. To name a few, for the case of regular fractional
Brownian motions that corresponds to $H>1/2$ (that will be the considered case
in this paper), we refer to \cite{GrAnh}, \cite{GuLeTin}, \cite{MasNua03},
\cite{MasSchm04}, \cite{NuaRas02} and \cite{TinTuVi}. To our knowledge, only
in one of these papers, to be more precise in \cite{MasNua03}, there is one
result corresponding to existence and not uniqueness of solutions, but
assuming between other assumptions that the initial condition is more regular
(see Section \ref{existence} for a deeper discussion). \newline

In spite of the non-uniqueness of solutions, the assumptions on the nonlinear
mapping $G$ will make it possible for us to define a pathwise integral against
$\omega$, based on a generalization of the Young integral given by the
so-called fractional derivatives, which does not produce exceptional sets in
contrast to Ito integral. We will take advantage of this property in order to
analyze the dynamical system generated by the solutions of the above problem
without making a previous transformation of it into a random equation, that is
to say, we will always work with (\ref{P}). In the univalued setting there
have been some progresses in this direction, namely, once the existence and
uniqueness of solutions have been established, the investigations of the
generation of the (univalued) random dynamical system have been carried out,
see \cite{CGGSch14}, \cite{GLS10}, \cite{GMSch} among others. This property
opens the door to investigate the longtime behavior of solutions by studying
the random attractor, invariant stable/unstable manifolds, random fixed
points, etc., using for that the theory of random dynamical systems, see the
monograph \cite{Arn98} for a comprehensive description. In the univalued case
the reader could find some papers dealing with these objects, as \cite{GGSch}
and \cite{GMSch}, where the random attractor have been studied, or
\cite{DGNSch17} and \cite{GNSch16} where the exponential stability of the
trivial solution have been considered.\newline

No matter the univalued or the multivalued settings, the random dynamical
system consists of an ergodic metric dynamical system, that models the noise,
and a cocycle mapping given by the solution(s) associated to an initial
condition. There have been some papers in which the ergodicity of the metric
dynamical system associated with the fractional Brownian motion has been
investigated, see for instance \cite{GSch11} and \cite{MasSchm04}. But as far
as the multivalued setting is concerned these previous investigations cannot
be applied, since the set $\Omega$ where $\omega$ belongs should be, on the
one hand, a space of H\"older continuous functions allowing the construction
of the stochastic integral, and, on the other hand, a separable space that
makes it possible to prove the measurability of the cocycle. Therefore, this
paper contains the construction of a new ergodic metric dynamical system
modeling the fractional Brownian motion, valid for any Hurst parameter, with
the mentioned separability property. \newline

The paper is structured as follows. In Section \ref{preli} we introduce the
integral with integrator given by a H\"{o}lder function with H\"{o}lder index
bigger than 1/2, as well as its main properties. Section \ref{existence}
addresses the existence of solutions to (\ref{P}) and for that we shall apply
Schauder's fixed point theorem. Further, Section \ref{MDS} is concerned with a
metric dynamical systems that model the fractional Brownian motion. In
particular, we will construct an ergodic metric dynamical system with the
property that the set space is separable. Finally, Section \ref{mnds}
investigates the generation of a multivalued nonautonomous dynamical system by
the set of all solutions to (\ref{P}). This multivalued mapping will be shown
to be measurable, and hence it is a multivalued random dynamical system, for
which we need to study its upper semicontinuity with respect to all its
variables, a result that relies upon the separability of the metric dynamical
system built in Section \ref{MDS}. The paper finally includes an example of a
stochastic parabolic partial differential equation for which our theory can be applied.

\section{Preliminaries}

\label{preli}

In this section we introduce the assumptions of the different terms of
equation (\ref{P}) and give the definition and some properties of the integral
for H{\"o}lder continuous integrators with H{\"o}lder exponent bigger
than$\,1/2$, for which we borrow the construction carried out recently in
\cite{CGGSch14}.\newline

Throughout this paper ($V$, $\|\cdot\|$, $(\cdot,\cdot)$) is a separable
Hilbert space.

Assume that $-A$ is a strictly positive and symmetric operator with a compact
inverse, generating an analytic semigroup $S$ on $V$. In particular the spaces
$V_{\delta}:=D((-A)^{\delta})$ with norm $\|\cdot\|_{V_{\delta}}$ for
$\delta\ge0$ are well-defined. Note that $V=V_{0}$ and $V_{\delta_{1}}$ is
compactly embedded in $V_{\delta_{2}}$ for $\delta_{1} \geq\delta_{2}\geq0$.
Denote by $(e_{i})_{i\in\mathbb{N} }$ the complete orthonormal base in $V$
generated by the eigenelements of $-A$ with associated eigenvalues
$(\lambda_{i})_{i\in\mathbb{N} }$.\newline

Let $L(V_{\delta},V_{\gamma})$ denote the space of continuous linear operators
from $V_{\delta}$ into $V_{\gamma}$. Then there exists a constant $c_{S}>0$
such that
\begin{equation}
\Vert S(t)\Vert_{L(V,V_{\gamma})}=\Vert(-A)^{\gamma}S(t)\Vert_{L(V)}\leq
c_{S}e^{-\lambda t}t^{-\gamma}\qquad\text{for }\gamma>0, \label{eq1}%
\end{equation}%
\begin{equation}
\Vert S(t)-\mathrm{id}\Vert_{L(V_{\sigma},V_{\theta})}\leq c_{S}%
t^{\sigma-\theta},\quad\text{for }\theta\geq0,\quad\sigma\in\lbrack
\theta,1+\theta]. \label{eq2}%
\end{equation}
The constant $c_{S}$ may depend on $t$ as well as the parameters $\gamma$,
$\sigma$ and $\theta$, and can change from line to line. In general, we will
only emphasize the dependence on the semigroup. Moreover, in (\ref{eq1})
$\lambda$ is a positive constant such that $\lambda\leq\lambda_{1}$. From
these two inequalities it is straightforward to derive that
\begin{equation}%
\begin{split}
&  \Vert S(t-r)-S(t-q)\Vert_{L(V_{\delta},V_{\gamma})}\leq c(r-q)^{\alpha
}(t-r)^{-\alpha-\gamma+\delta},\\
&  \Vert S(t-r)-S(s-r)-S(t-q)+S(s-q)\Vert_{L(V)}\leq c(t-s)^{\beta
}(r-q)^{\gamma}(s-r)^{-(\beta+\gamma)},
\end{split}
\label{eq30}%
\end{equation}
for $0\leq q\leq r\leq s\leq t$. As usual, $L(V)$ denotes the space $L(V,V)$.

Let $C^{\beta}([T_{1},T_{2}],V)$ be the Banach space of H{\"o}lder continuous
functions with exponent $\beta>0$ having values in $V$. A norm on this space
is given by%
\[
\|u\|_{\beta}=\|u\|_{\beta,T_{1},T_{2}}=\|u\|_{\infty, T_{1},T_{2}}+\left|  \!
\left|  \! \left|  u \right|  \! \right|  \! \right|  _{\beta,T_{1},T_{2}%
}:=\sup_{s\in[T_{1},T_{2}]}\|u(s)\|+\sup_{T_{1}\le s<t\le T_{2}}%
\frac{\|u(t)-u(s)\|}{(t-s)^{\beta}}.
\]
$C([T_{1},T_{2}],V)$ denotes the space of continuous functions on
$[T_{1},T_{2}]$ with values in $V$ with finite supremum norm. Since the
properties on the semigroup do not ensure H\"older continuity at zero, we will
work with a modification of the spaces of H\"older continuous functions, in
the sense that the considered norm will be given by
\[
\|u\|_{\beta,\beta}=\|u\|_{\beta,\beta,T_{1},T_{2}}=\|u\|_{\infty, T_{1}%
,T_{2}}+\left|  \! \left|  \! \left|  u \right|  \! \right|  \! \right|
_{\beta,\beta,T_{1},T_{2}}:=\|u\|_{\infty, T_{1},T_{2}}+\sup_{T_{1}< s<t\le
T_{2}}(s-T_{1})^{\beta}\frac{\|u(t)-u(s)\|}{(t-s)^{\beta}}%
\]
and denote by $C^{\beta}_{\beta}([T_{1},T_{2}],V)$ the set of functions $u \in
C([T_{1},T_{2}],V)$ such that $\|u\|_{\beta,\beta} < \infty. $ It is known
that $C^{\beta}_{\beta}([T_{1},T_{2}],V)$ is a Banach space, see
\cite{CGGSch14} and \cite{lunardi}.

For every $\rho>0$ we can consider the equivalent norm%
\begin{align*}
\|u\|_{\beta,\beta;\rho}=\|u\|_{\beta,\beta;\rho,T_{1},T_{2}}  &  =\sup
_{s\in[T_{1},T_{2}]}e^{-\rho(s-T_{1})}\|u(s)\|\\
&  +\sup_{T_{1}< s<t\le T_{2}}(s-T_{1})^{\beta}e^{-\rho(t-T_{1})}%
\frac{\|u(t)-u(s)\|}{(t-s)^{\beta}}.
\end{align*}

Consider now the separable Hilbert space $L_{2}(V)$ of Hilbert-Schmidt
operators from $V$ into $V$ with the usual norm $\|\cdot\|_{L_{2}(V)}$ and
inner product $(\cdot,\cdot)_{L_{2}(V)}$.

In order to define the integral with respect to a H\"older integrator $\omega
$, we introduce fractional derivatives, see \cite{Samko} for a comprehensive
description of this subject. More precisely, we define the left sided
fractional derivative of order $\alpha\in(0,1)$ of a sufficiently regular
function $g$ and the right sided fractional derivative of order $1-\alpha$ of
$\omega_{t-}(\cdot):=\omega(\cdot)-\omega(t)$, given by the expressions
\begin{align*}
\label{fractder}%
\begin{split}
D_{{s}+}^{\alpha}g[r]=  &  \frac{1}{\Gamma(1-\alpha)}\bigg(\frac
{g(r)}{(r-s)^{\alpha}}+\alpha\int_{s}^{r}\frac{g(r)-g(q)}{(r-q)^{1+\alpha}%
}dq\bigg),\\
D_{{t}-}^{1-\alpha} \omega_{{t}-}[r]=  &  \frac{(-1)^{1-\alpha}}{\Gamma
(\alpha)} \bigg(\frac{\omega(r)-\omega(t)}{(t-r)^{1-\alpha}}+(1-\alpha
)\int_{r}^{t}\frac{\omega(r)-\omega(q)}{(q-r)^{2-\alpha}}dq\bigg),
\end{split}
\end{align*}
where $0\le s\le r\le t$ and $\Gamma(\cdot)$ denotes the Gamma function. We
assume that $1-\beta^{\prime}<\alpha<\beta$ and let
\[
[s,t]\ni r\mapsto g(r)\in L_{2}(V),\quad[s,t]\ni r\mapsto\omega(r)\in V
\]
be measurable functions such that $g\in C^{\beta}_{\beta}([0,T], L_{2}(V))$,
$\omega\in C^{\beta^{\prime}} ([0,T],V)$ and
\[
r\mapsto\|D_{s+}^{\alpha}g[r]\|_{L_{2}(V)}\|D_{t-}^{1-\alpha}\omega[r]\|
\]
is Lebesgue integrable. Then for $0\le s\le r\le t \le T$ we can define
\begin{equation}
\label{integral}\int_{s}^{t} g(r)d\omega(r):=(-1)^{\alpha}\sum_{j\in\mathbb{N}
}\bigg(\sum_{i\in\mathbb{N} }\int_{s}^{t} D_{s+}^{\alpha}(e_{j},g(\cdot
)e_{i})_{V}[r]D_{t-}^{1-\alpha}(e_{i},\omega(\cdot))_{ V}[r]dr\bigg)e_{j}.
\end{equation}

This integral is well--defined and it is given by a generalization of the
pathwise integral introduced by Z\"ahle \cite{Zah98}, which was given as an
extension of the Young integral (see \cite{You36}).\newline

In the following result we collect some interesting properties of the integral
with integrator $\omega$. In the sequel, we assume the following constraints
for the different parameters:
\[
1/2<\beta<\beta^{\prime}, \; 1-\beta^{\prime}< \alpha< \beta.
\]
Further, we will also consider $\beta^{\prime}<H$, where $H$ denotes the Hurst
parameter of a fractional Brownian motion.

\begin{lemma}
\label{l3} Let $T>0$, $\omega\in C^{\beta^{\prime}}([0,T],V)$ and $g\in
C^{\beta}_{\beta}([s,t],L_{2}(V))$, for $0\leq s\leq t\leq T$. Then
(\ref{integral}) is well-defined and satisfies the following properties:

\begin{enumerate}
\item[(i)] The norm can be estimated as
\begin{align*}
\bigg\|\int_{s}^{t}g(r)d\omega(r)\bigg\|  &  \leq\int_{s}^{t}\Vert
D_{s+}^{\alpha}g[r]\Vert_{L_{2}(V)}\Vert D_{t-}^{1-\alpha}\omega_{t-}[r]\Vert
dr\\
&  \leq c\Vert g\Vert_{\beta,\beta,s,t}\left\vert \!\left\vert \!\left\vert
\omega\right\vert \!\right\vert \!\right\vert _{\beta^{\prime},s,t}%
(t-s)^{{\beta^{\prime}}},
\end{align*}
where $c$ is a positive constant depending only on $t,\,\beta,\,\beta^{\prime
}$ (see \cite{CGGSch14}).

\item[(ii)] The integral is additive (see \cite{Zah98} Theorem 2.5):
\[
\int_{s}^{\tau} g(r) d\omega(r)+\int_{\tau}^{t} g(r) d\omega(r)=\int_{s}^{t}
g(r) d\omega(r)\quad\text{for }s<\tau<t.
\]

\item[(iii)] For any $\tau\in\mathbb{R} $ it yields
\[
\int_{s}^{t} g(r) d\omega(r)=\int_{s-\tau}^{t-\tau} g(r+\tau) d\theta_{\tau
}\omega(r),
\]
(see \cite{GLS10} Lemma 5), where
\begin{equation}
\label{shift}\theta_{t}\omega(\cdot):=\omega(t+\cdot)-\omega(t)
\end{equation}
is known as the Wiener shift.
\end{enumerate}
\end{lemma}

To end this section we introduce the nonlinear mappings of (\ref{P}). We
assume that $F:V\mapsto V$ is a continuous function with at most linear
growth, that is, there exist $c_{F},\,L_{F}>0$ such that
\[
\|F(u)\| \leq c_{F} +L_{F}\|u\|,\quad\text{for } u\in V,
\]
and $G:V \mapsto L_{2}(V)$ is Lipschitz continuous with Lipschitz constant
$L_{G}$. Denoting $c_{G}=\|G(0)\|_{L_{2}(V)}$ we then have
\[
\|G(u)\|_{L_{2}(V)} \leq c_{G} +L_{G}\|u\|,\quad\text{for } u\in V.
\]

%%%%%%%%%%%%%%%%%%%%%%%%%%%%%%%%%%

\section{Existence of solutions}

\label{existence}

Our main goal in this section is to study the existence of solutions to
(\ref{P}). First of all, we introduce the exact definition of a solution to
that problem.

\begin{definition}
Under the assumptions on $A$, $F$ and $G$ described in Section \ref{preli},
given $T>0$ we say that $u$ is a mild pathwise solution to (\ref{P})
corresponding to the initial condition $u_{0}\in V$ if $u\in C^{\beta}_{\beta
}([0,T],V)$ and satisfies for $t\in[0,T]$ the equation
\begin{equation}
\label{Pm}u(t)=S(t)u_{0}+\int_{0}^{t}S(t-r)F(u(r))dr+\int_{0}^{t}%
S(t-r)G(u(r))d\omega.
\end{equation}

\end{definition}

The first integral on the right hand side is a standard Lebesgue integral,
while the integral with respect to $\omega$ is interpreted in the sense of the
previous section.

\medskip

Existence and uniqueness of solutions for this kind of equations has been
investigated in \cite{CGGSch14} and in \cite{MasNua03}. In both references the
nonlinear mapping $G$ is assumed to be twice Fr\'echet differentiable with
bounded first and second derivatives, from which the uniqueness of solutions
can be derived. However, in this article we get rid of such strong regularity
and only assume $G$ to be Lipschitz continuous, which is not sufficient to
ensure the uniqueness of solutions.

We would like to mention that Theorem 3.1 in \cite{MasNua03} also deals with
existence-- and not uniqueness-- of a problem like the above one. To be more
precise, under Lipschitz regularity of $G$, assuming that the initial
condition is more regular (belonging to a particular space $V_{\kappa}$
instead of $V$), the authors prove existence of solutions in $W^{\alpha,
\infty}([0,T],V)$, the space of measurable functions $x:[0,T] \mapsto V$ such
that
\[
\|x\|_{\alpha,\infty}:=\sup_{t\in[0,T]} \bigg(\|x(t)\|+\int_{0}^{t}
\frac{\|x(t)-x(s)\|}{(t-s)^{1+\alpha}} ds \bigg)<\infty.
\]
Recall that $\alpha<\beta$, which in turn implies that $C^{\beta}_{\beta
}([0,T],V) \subset W^{\alpha, \infty}([0,T],V).$

In our setting, we shall apply Schauder's theorem to establish the existence
of solutions to (\ref{Pm}). To this end, for $\omega\in C^{\beta^{\prime}%
}([0,T],V)$ and $u_{0}\in V$ consider the operators
\[
\mathcal{T}(\cdot,\omega,u_{0}):C_{\beta}^{\beta}([0,T],V)\mapsto C_{\beta
}^{\beta}([0,T],V),
\]%
\[
\mathcal{T}^{I}(\cdot,\omega):C_{\beta}^{\beta}([0,T],V)\mapsto C_{\beta
}^{\beta}([0,T],V)
\]
defined by
\[
\mathcal{T}(u,\omega,u_{0})(t)=S(t)u_{0}+\int_{0}^{t} S(t-r)F(u(r))d
r+\int_{0}^{t}S(t-r)G(u(r))d\omega,
\]%
\[
\mathcal{T}^{I}(u,\omega)(t)=\int_{0}^{t} S(t-r)F(u(r))dr+\int_{0}%
^{t}S(t-r)G(u(r))d\omega.
\]

The application of Schauder's theorem will be based on suitable estimates
given in the following lemmas. In the different proofs, $c$ will denote a
generic constant that may differ from line to line. Sometimes we will write
$c_{T}$ or $c_{S}$ when we want to stress the dependence on $T$ or on the semigroup.

We start by stating the following technical result:

\begin{lemma}
\label{l16} Let $a>-1,\,b>-1$ and $a+b\ge-1,\,d> 0$ and $t \in[0,T]$. If for
$\rho>0$ we define%
\[
H(\rho):=\sup_{t \in[0,T]}t^{d}\int_{0}^{1}e^{-\rho t(1-v)}v^{a}(1-v)^{b}dv,
\]
then we have that $\lim_{\rho\to\infty}H(\rho)=0$.
\end{lemma}

The proof can be found in \cite{CGGSch14}, where it can be checked that
$H(\rho)$ is related to the Kummer or hypergeometric function.\newline

For the sake of presentation we denote $\left\vert \!\left\vert \!\left\vert
\omega\right\vert \!\right\vert \!\right\vert _{\beta^{\prime},0,T}$ by
$\left\vert \!\left\vert \!\left\vert \omega\right\vert \!\right\vert
\!\right\vert _{\beta^{\prime}}$ (and $\|\mathcal{T} (u,\omega,u_{0}%
)\|_{\beta,\beta;\rho, 0,T}$ by $\|\mathcal{T} (u,\omega,u_{0})\|_{\beta
,\beta;\rho}$), when the time interval does not produce any confusion.

\begin{lemma}
\label{l6} For any $T>0$ there exists a $c_{T}>0$ (that also depends on the
constants related to $F$, $G$ and the semigroup) such that for $\omega\in
C^{\beta^{\prime}}([0,T],V)$, $u_{0}\in V$ and $u\in C^{\beta}_{\beta
}([0,T],V)$
\begin{equation}
\label{eq20}\|\mathcal{T} (u,\omega,u_{0})\|_{\beta,\beta;\rho}\le
c_{S}\|u_{0}\|+c_{T} \left|  \! \left|  \! \left|  \omega\right|  \! \right|
\! \right|  _{\beta^{\prime}}K(\rho)(1+\|u\|_{\beta,\beta;\rho})
\end{equation}
where $K(\rho)$ is such that $\lim_{\rho\to\infty}K(\rho)=0$ and $c_{S}\geq1$
is a constant depending on the semigroup $S$.
\end{lemma}

\begin{proof}

Despite the fact that a quite similar result was proved in \cite{CGGSch14}
(but in that paper there was not any drift), for the sake of completeness we
give the proof here. First of all, according to (\ref{eq30}) we have that
\begin{align*}
\|S(\cdot)u_{0}\|_{\beta,\beta;\rho}  &  = \sup_{t\in[0,T]}e^{-\rho
t}\|S(t)u_{0}\|+\sup_{0< s<t\le T}s^{\beta}e^{-\rho t} \frac{\|S(t)u_{0}%
-S(s)u_{0}\|}{(t-s)^{\beta}}\\
&  \le c_{S} \|u_{0}\|+c_{S} \sup_{0< s<t\le T}s^{\beta}e^{-\rho t} \frac{
s^{-\beta} e^{-\lambda s} (t-s)^{\beta}}{(t-s)^{\beta}}\|u_{0}\|\\
&  \le c_{S} \|u_{0}\|.
\end{align*}
Note that the last constant $c_{S}$ above is bigger than one, which follows
from the fact that in particular
\[
\|u_{0}\|=\|S(0)u_{0}\|\leq\sup_{t\in[0,T]}e^{-\rho t}\|S(t)u_{0}\|\le c_{S}
\|u_{0}\|.
\]

On the other hand,
\begin{align}
\label{equ31b}%
\begin{split}
\bigg\| \int_{0}^{\cdot}S(\cdot-r) F(u(r)) dr \bigg\|_{\beta,\beta;\rho}  &
\leq\sup_{t\in[0,T]}e^{-\rho t}\bigg\|\int_{0}^{t}S(t-r)F(u(r))dr \bigg\|\\
&  + \sup_{0< s<t\leq T} \frac{s^{\beta}e^{-\rho t}}{{(t-s)^{\beta}}}
\bigg\|\int_{s}^{t}S(t-r)F(u(r))dr\bigg\|\\
&  + \sup_{0< s<t\leq T} \frac{s^{\beta}e^{-\rho t}}{{(t-s)^{\beta}}%
}\bigg\|\int_{0}^{s}(S(t-r)-S(s-r))F(u(r))dr\bigg\|.
\end{split}
\end{align}
Due to the at most linear growth of $F$,
\begin{align*}
\sup_{t\in[0,T]}e^{-\rho t}\bigg\|\int_{0}^{t}S(t-r)F(u(r))dr\bigg\|  &  \leq
c_{S,F} \sup_{t\in[0,T]}e^{-\rho t}\int_{0}^{t}(1+\|u(r)\|)dr\\
&  =c_{S,F}\sup_{t\in\lbrack0,T]}\left(  e^{-\rho t}t+\left\Vert u\right\Vert
_{\beta,\beta,\rho}\frac{1-e^{-\rho t}}{\rho}\right) \\
&  \leq c_{S,F}k_{1}(\rho)\left(  1+\left\Vert u\right\Vert _{\beta,\beta
,\rho}\right)  ,
\end{align*}
where $k_{1}(\rho)=\frac{1}{\rho}$. Above we have used that $\max_{t\geq
0}e^{-\rho t}t=\frac{e^{-1}}{\rho}\leq\frac{1}{\rho}.$ Moreover,
\begin{align*}
\frac{s^{\beta}e^{-\rho t}}{{(t-s)^{\beta}}} \bigg\|\int_{s}^{t}%
S(t-r)F(u(r))dr\bigg\|  &  \leq c_{S,F,T} \bigg(e^{-\rho t}(t-s)^{1-\beta
}+\|u\|_{\beta,\beta;\rho} \frac{\int_{s}^{t} e^{-\rho(t-r)}dr }{(t-s)^{\beta
}}\bigg)\\
&  \leq c_{S,F,T} \bigg(e^{-\rho t}(t-s)^{1-\beta}+\|u\|_{\beta,\beta;\rho}
\frac{1}{\rho^{1-\beta}} \sup_{x>0} \frac{1-e^{-x}}{x^{\beta}}\bigg)\\
&  \leq c_{S,F,T} k_{2}(\rho)(1+\|u\|_{\beta,\beta;\rho}), \; \text{for all }
0<s<t\leq T,
\end{align*}
where $k_{2}(\rho)=\frac{1}{\rho^{1-\beta}}$. Note that now we have used that
\[
\max_{t\geq0} e^{-\rho t}(t-s)^{1-\beta}\leq e^{\beta-1} \frac{(1-\beta
)^{1-\beta}}{\rho^{1-\beta}}\leq\frac{1}{\rho^{1-\beta}}.
\]
For the last term we have
\begin{align*}
\frac{s^{\beta}e^{-\rho t}}{{(t-s)^{\beta}}}\bigg\|\int_{0}^{s}%
(S(t-r)-S(s-r))F(u(r))dr\bigg\|  &  \leq c_{S} \frac{s^{\beta}e^{-\rho t}%
}{{(t-s)^{\beta}}}\int_{0}^{s}\frac{(t-s)^{\beta}}{(s-r)^{\beta}} (c_{F}%
+L_{F}\|u(r)\|)dr\\
&  \leq c_{S,F} (1+\|u\|_{\beta,\beta;\rho}) s^{\beta} \int_{0}^{s}
e^{-\rho(s-r)} (s-r)^{-\beta} dr\\
&  \leq c_{S,F,T} k_{3}(\rho)(1+\|u\|_{\beta,\beta;\rho}), \; \text{for all }
0<s<t\leq T,
\end{align*}
with $k_{3}(\rho)$ defined as in Lemma \ref{l16}. For the stochastic integral
we have a similar splitting than before, namely
\begin{align}
\label{equ31}%
\begin{split}
\bigg\| \int_{0}^{\cdot}S(\cdot-r) G(u(r)) d\omega\bigg\|_{\beta,\beta;\rho}
&  \leq\sup_{t\in[0,T]}e^{-\rho t}\bigg\|\int_{0}^{t}S(t-r)G(u(r))d\omega
\bigg\|\\
&  + \sup_{0< s<t\leq T} \frac{s^{\beta}e^{-\rho t}}{{(t-s)^{\beta}}}
\bigg\|\int_{s}^{t}S(t-r)G(u(r))d\omega\bigg\|\\
&  + \sup_{0< s<t\leq T} \frac{s^{\beta}e^{-\rho t}}{{(t-s)^{\beta}}%
}\bigg\|\int_{0}^{s}(S(t-r)-S(s-r))G(u(r))d\omega\bigg\|.
\end{split}
\end{align}
From the definition of the fractional derivative of order $1-\alpha$ it
follows easily that
\[
\|D_{t-}^{1-\alpha}\omega[r]\|\le c \left|  \! \left|  \! \left|
\omega\right|  \! \right|  \! \right|  _{\beta^{\prime}}(t-r)^{\alpha
+\beta^{\prime}-1},
\]
hence, thanks also to (\ref{eq30}) we obtain
\begin{align}
\label{l11}%
\begin{split}
&  s^{\beta}e^{-\rho t}\bigg\|\int_{s}^{t}S(t-r)G(u(r))d\omega\bigg\|\\
&  \le c s^{\beta}e^{-\rho t} \int_{s}^{t}\bigg(\frac{\|S(t-r)\|_{L(V)}%
\|G(u(r))\|_{L_{2}(V)}}{(r-s)^{\alpha}}\\
&  +\int_{s}^{r}\frac{\|S(t-r)-S(t-q)\|_{L(V)}\|G(u(r))\|_{L_{2}(V)}%
}{(r-q)^{1+\alpha}}dq\\
&  +\int_{s}^{r}\frac{\|S(t-q)\|_{L(V)}\|G(u(r))-G(u(q))\|_{L_{2}(V)}%
}{(r-q)^{1+\alpha}}dq \bigg)\left|  \! \left|  \! \left|  \omega\right|  \!
\right|  \! \right|  _{\beta^{\prime}}(t-r)^{\alpha+\beta^{\prime}-1}dr\\
&  \leq c_{S,T}\left|  \! \left|  \! \left|  \omega\right|  \! \right|  \!
\right|  _{\beta^{\prime}}\bigg( \int_{s}^{t} e^{-\rho(t-r)}\frac{e^{-\rho r}
(c_{G}+L_{G}\|u(r)\|)}{(r-s)^{\alpha}}(t-r)^{\alpha+\beta^{\prime}-1}dr\\
&  +\int_{s}^{t}\int_{s}^{r}e^{-\rho(t-r)}\frac{ e^{-\rho r}(c_{G}%
+L_{G}\|u(r)\|)(r-q)^{\beta}}{(t-r)^{\beta}(r-q)^{1+\alpha}}dq(t-r)^{\alpha
+\beta^{\prime}-1}dr\\
&  +\int_{s}^{t}\int_{s}^{r}e^{-\rho(t-r)}\frac{ e^{-\rho r}L_{G}%
\|u(r)-u(q)\|q^{\beta}(r-q)^{\beta}}{(r-q)^{1+\alpha}q^{\beta}(r-q)^{\beta}%
}dq(t-r)^{\alpha+\beta^{\prime}-1}dr\bigg)\\
&  \le c_{S,G,T}\left|  \! \left|  \! \left|  \omega\right|  \! \right|  \!
\right|  _{\beta^{\prime}}(t-s)^{\beta^{\prime}} (1+\|u\|_{\beta,\beta;\rho
})\int_{s}^{t}e^{-\rho(t-r)}(r-s)^{-\alpha}(t-r)^{\alpha-1}dr\\
&  +c_{S,G,T}\left|  \! \left|  \! \left|  \omega\right|  \! \right|  \!
\right|  _{\beta^{\prime}}(1+\|u\|_{\beta,\beta;\rho})\int_{s}^{t}%
e^{-\rho(t-r)}(r-s)^{\beta-\alpha}(t-r)^{\alpha+\beta^{\prime}-1-\beta}dr\\
&  + c_{S,G,T}\left|  \! \left|  \! \left|  \omega\right|  \! \right|  \!
\right|  _{\beta^{\prime}}(t-s)^{\beta^{\prime}} \|u\|_{\beta,\beta;\rho}%
\int_{s}^{t}e^{-\rho(t-r)}(r-s)^{-\alpha}(t-r)^{\alpha-1}dr.
\end{split}
\end{align}

Performing a change of variable, it is easy to see that
\begin{align*}
&  (t-s)^{\beta^{\prime}}\int_{s}^{t}e^{-\rho(t-r)}(r-s)^{-\alpha
}(t-r)^{\alpha-1}dr\\
=  &  (t-s)^{\beta^{\prime}-\beta}(t-s)^{\beta}\int_{0}^{1} e^{-\rho
(t-s)(1-v)} v^{-\alpha}(1-v)^{\alpha-1}dv=(t-s)^{\beta}k_{4}(\rho)
\end{align*}
with $\lim_{\rho\to\infty} k_{4}(\rho)=0$, taking in Lemma \ref{l16}
$a=-\alpha,\,b=\alpha-1$, $d=\beta^{\prime}-\beta$ and $t-s$ as the
corresponding $t$ there. The second integral on the right hand side may be
rewritten in the same way, since%
\begin{align*}
\int_{s}^{t}e^{-\rho(t-r)}(r-s)^{\beta-\alpha}(t-r)^{\alpha+\beta^{\prime
}-1-\beta}dr \leq(t-s)^{\beta^{\prime}}\int_{s}^{t}e^{-\rho(t-r)}%
(r-s)^{-\alpha}(t-r)^{\alpha-1}dr.
\end{align*}
Therefore, from (\ref{l11}) we obtain
\begin{align*}
s^{\beta}  &  e^{-\rho t}\bigg\|\int_{s}^{t}S(t-r)G(u(r))d\omega\bigg\| \le
c_{S,G,T} \left|  \! \left|  \! \left|  \omega\right|  \! \right|  \! \right|
_{\beta^{\prime}} (t-s)^{\beta}k_{4}(\rho) (1+\|u\|_{\beta,\beta;\rho}).
\end{align*}

In a similar manner than before, for the first expression on the right hand
side of \eqref{equ31} we obtain
\[
e^{-\rho t}\bigg\|\int_{0}^{t}S(t-r)G(u(r))d\omega\bigg\|\ \leq c_{S,G,T}
\left|  \! \left|  \! \left|  \omega\right|  \! \right|  \! \right|
_{\beta^{\prime}}k_{5}(\rho) (1+\|u\|_{\beta,\beta;\rho}),
\]
with $\lim_{\rho\to\infty} k_{5}(\rho)=0$. Finally, for the third term on the
right hand side of (\ref{equ31}) we follow similar steps than above when
deriving the inequality (\ref{l11}). Indeed, for $0<\alpha<\alpha^{\prime}<1$
such that $\alpha^{\prime}+\beta<\alpha+\beta^{\prime}$, applying (\ref{eq30})
we have
\[%
\begin{split}
s^{\beta}e^{-\rho t}  &  \bigg\|\int_{0}^{s}(S(t-r)-S(s-r))G(u(r))d\omega
\bigg\|\\
&  \leq c(t-s)^{\beta}\left\vert \!\left\vert \!\left\vert \omega\right\vert
\!\right\vert \!\right\vert _{\beta^{\prime}}T^{\beta}\bigg(\int_{0}
^{s}e^{-\rho(t-r)} \frac{e^{-\rho r}(c_{G}+L_{G}\Vert u(r)\Vert)}{r^{\alpha
}(s-r)^{\beta} }(s-r)^{\alpha+\beta^{\prime}-1}dr\\
&  +\int_{0}^{s}\int_{0}^{r}e^{-\rho(t-r)}\frac{e^{-\rho r}(c_{G}+L_{G}\Vert
u(r)\Vert)(r-q)^{\alpha^{\prime}}}{(s-r)^{\alpha^{\prime}+\beta}
(r-q)^{1+\alpha}}dq(s-r)^{\alpha+\beta^{\prime}-1}dr\\
&  +\int_{0}^{s}\int_{0}^{r}e^{-\rho(t-r)}\frac{e^{-\rho r}L_{G}\Vert
u(r)-u(q)\Vert q^{\beta}(r-q)^{\beta}}{(s-r)^{\beta}(r-q)^{1+\alpha}q^{\beta
}(r-q)^{\beta} }dq(s-r)^{\alpha+\beta^{\prime}-1}dr\bigg)\\
&  \leq c_{S,G,T}(t-s)^{\beta}\left\vert \!\left\vert \!\left\vert
\omega\right\vert \!\right\vert \!\right\vert _{\beta^{\prime}}(1+\Vert
u\Vert_{\beta,\beta;\rho})\int_{0}^{s}e^{-\rho(t-r)} r^{-\alpha}%
(s-r)^{\alpha+\beta^{\prime}-1-\beta}dr\\
&  +c_{S,G,T}(t-s)^{\beta}\left\vert \!\left\vert \!\left\vert \omega
\right\vert \!\right\vert \!\right\vert _{\beta^{\prime}}(1+\Vert
u\Vert_{\beta,\beta;\rho})\int_{0}^{s}e^{-\rho(t-r)}r^{\alpha^{\prime}-\alpha
}(s-r)^{\alpha+\beta^{\prime}-1-\alpha^{\prime}-\beta}dr\\
&  +c_{S,G,T}(t-s)^{\beta}\left\vert \!\left\vert \!\left\vert \omega
\right\vert \!\right\vert \!\right\vert _{\beta^{\prime}}\Vert u\Vert
_{\beta,\beta;\rho}\int_{0}^{s}e^{-\rho(t-r)}r^{-\alpha} (s-r)^{\alpha
-\beta+\beta^{\prime}-1}dr.
\end{split}
\]
Collecting all the above estimates the proof is complete.
\end{proof}

\begin{lemma}
\label{l0} For $\omega\in C^{\beta^{\prime}}([0,T],V)$ and $u\in C_{\beta
}^{\beta}([0,T],V)$ the mapping $\mathcal{T}^{I}(u,\omega)(t)\in V_{\delta}$
for every $t\geq0$ and $\delta\in\lbrack0,\beta^{\prime}).$ Moreover, there
exists a constant $c$ depending on $S$, $F$ and $G$ such that
\[
\Vert\mathcal{T}^{I}(u,\omega)(t)\Vert_{V_{\delta}}\leq c(t^{\beta^{\prime
}-\delta}\left\vert \!\left\vert \!\left\vert \omega\right\vert \!\right\vert
\!\right\vert _{\beta^{\prime}}+t^{1-\delta})(1+\Vert u\Vert_{\beta,\beta}).
\]
\end{lemma}

\begin{proof}
For the deterministic integral we directly obtain
\begin{align*}
\bigg\|\int_{0}^{t}S(t-r)F(u(r))dr\bigg\|_{V_{\delta}}  &  \leq c_{S,F}
\int_{0}^{t} (t-r)^{-\delta}(1+\|u(r)\|)dr \leq c_{S,F} t^{1-\delta
}(1+\|u\|_{\beta,\beta}).
\end{align*}
For the stochastic integral we have
\[%
\begin{split}
&  \bigg\|\int_{0}^{t}S(t-r)G(u(r))d\omega\bigg\|_{V_{\delta}}\\
&  \leq c\int_{0}^{t}\bigg(\frac{\Vert S(t-r)\Vert_{L(V,V_{\delta})}\Vert
G(u(r))\Vert_{L_{2}(V)}}{r^{\alpha}}\\
&  +\int_{0}^{r}\frac{\Vert S(t-r)-S(t-q)\Vert_{L(V,V_{\delta})}\Vert
G(u(r))\Vert_{L_{2}(V)}}{(r-q)^{1+\alpha}}dq\\
&  +\int_{0}^{r}\frac{\Vert S(t-q)\Vert_{L(V,V_{\delta})}\Vert
G(u(r))-G(u(q))\Vert_{L_{2}(V)}}{(r-q)^{1+\alpha}}dq\bigg)\left\vert
\!\left\vert \!\left\vert \omega\right\vert \!\right\vert \!\right\vert
_{\beta^{\prime}}(t-r)^{\alpha+\beta^{\prime}-1}dr\\
&  \leq c_{S}\left\vert \!\left\vert \!\left\vert \omega\right\vert
\!\right\vert \!\right\vert _{\beta^{\prime}}\bigg(\int_{0}^{t}\frac
{(c_{G}+L_{G}\Vert u(r)\Vert)}{r^{\alpha}}(t-r)^{\alpha+\beta^{\prime
}-1-\delta}dr\\
&  +\int_{0}^{t}\int_{0}^{r}\frac{(c_{G}+L_{G}\Vert u(r)\Vert)(r-q)^{\alpha
^{\prime}}}{(t-r)^{\alpha^{\prime}+\delta}(r-q)^{1+\alpha}}dq(t-r)^{\alpha
+\beta^{\prime}-1}dr\\
&  +\int_{0}^{t}\int_{0}^{r}\frac{L_{G}\Vert u(r)-u(q)\Vert q^{\beta
}(r-q)^{\beta}}{(r-q)^{1+\alpha}q^{\beta}(r-q)^{\beta}}dq(t-r)^{\alpha
+\beta^{\prime}-1-\delta}dr\bigg)\\
&  \leq c_{S,G}\left\vert \!\left\vert \!\left\vert \omega\right\vert
\!\right\vert \!\right\vert _{\beta^{\prime}}(1+\Vert u\Vert_{\beta,\beta
})t^{\beta^{\prime}-\delta}%
\end{split}
\]
where $\alpha^{\prime}$ has been chosen such that $0<\alpha<\alpha^{\prime}<1$
with $\alpha^{\prime}+\delta<\alpha+\beta^{\prime}$.
\end{proof}

\begin{corollary}
\label{l1} For $\omega\in C^{\beta^{\prime}}([0,T],V)$, $u_{0}\in V$ and $u\in
C_{\beta}^{\beta}([0,T],V)$ the mapping $\mathcal{T}(u,\omega,u_{0})(t)\in
V_{\delta}$ for every $t>0$ and $\delta\in\lbrack0,\beta^{\prime}).$
Moreover,
\[
\Vert\mathcal{T}(u,\omega,u_{0})(t)\Vert_{V_{\delta}}\leq c_{S}t^{-\delta
}\Vert u_{0}\Vert+c(t^{\beta^{\prime}-\delta}\left\vert \!\left\vert
\!\left\vert \omega\right\vert \!\right\vert \!\right\vert _{\beta^{\prime}%
}+t^{1-\delta})(1+\Vert u\Vert_{\beta,\beta}),
\]
where the constant $c$ depends on the semigroup and also the constants related
to $F$ and $G$.
\end{corollary}

\begin{proof}
For $t>0$ we trivially have $\|S(t)u_{0}\Vert_{V_{\delta}}\leq ct^{-\delta
}\Vert u_{0}\Vert.$ To conclude the result it suffices to take into account
Lemma \ref{l0}.
\end{proof}

We have also the following result:

\begin{theorem}
\label{t3} Denote by $B:=\bar{B}_{C_{\beta}^{\beta}}(0,R)$ and $\hat
B:=\bar{B}_{C^{\beta^{\prime}}}(0,R)$ the closed balls in $C_{\beta}^{\beta
}([0,T]; V)$ and $C^{\beta^{\prime}}([0,T],V)$, respectively, with radius $R$
and center $0$. Let $\mathcal{K} $ be a compact set in $V$. Then
$\mathcal{T}^{I}(B,\hat B)$ and $\mathcal{T}(B,\hat B,\mathcal{K} )$ are
relatively compact in $C^{\beta}([0,T],V)$ and $C_{\beta}^{\beta}([0,T],V)$, respectively.
\end{theorem}

\begin{proof}
For $u\in B$, $\omega\in\hat{B}$ and $T\geq t_{2}\geq t_{1}\geq0$ there exists
a $\gamma\in(\beta,\beta^{\prime})$ such that
\[
\Vert\mathcal{T}^{I}(u,\omega)(t_{2})-\mathcal{T}^{I}(u,\omega)(t_{1}%
)\Vert\leq c\left\vert \!\left\vert \!\left\vert \omega\right\vert
\!\right\vert \!\right\vert _{\beta^{\prime}}(1+\Vert u\Vert_{\beta,\beta
})(t_{2}-t_{1})^{\gamma},
\]
where $c$ is a positive constant that depends on the constants related to $S$,
$F$ and $G$, and $T$. The method to obtain the above estimate is similar to
the calculations in Lemma \ref{l6} setting $\rho=0$, see also Chen \textit{et
al.} \cite{CGGSch14}, hence we omit the proof here.

As a result, the set $\mathcal{T}^{I}(B,\hat B)$ is equicontinuous and bounded
in the space $C^{\gamma}([0,T],V)$.

On the other hand, in virtue of Lemma \ref{l0} we also have that
$\mathcal{T}^{I}(B,\hat B)(t) \in B_{\delta}$, for all $t\in\lbrack0,T]$,
where $B_{\delta}$ is a bounded set in $V_{\delta}$ with $0<\delta
<\beta^{\prime}$. As we know that $V_{\delta}\subset V$ compactly, we obtain
that the set $\mathcal{T}^{I}(B,\hat B)([0,T])$ belongs to a compact set of
$V$. By Lemma 4.5 in \cite{MasNua03} we have that $\mathcal{T}^{I}(B,\hat B)$
is relatively compact in $C^{\beta}([0,T],V)$ if $\beta\in(\alpha,\gamma)$ and
then in $C_{\beta}^{\beta}([0,T],V)$ as well, since $C^{\beta}([0,T],V)
\subset C_{\beta}^{\beta}([0,T],V)$ continuously.

Finally, let $u_{0}^{n}\in\mathcal{K} $. Then up to a subsequence $u_{0}%
^{n}\rightarrow u_{0}$ in $V$, and therefore
\begin{align*}
\left\Vert (S(t)-S(s))(u_{0}^{n}-u_{0})\right\Vert  &  \leq\left\Vert S(t-s)
-\mathrm{Id}\right\Vert _{L(V_{\beta},V)}\left\Vert S(s)\right\Vert _{L\left(
V,V_{\beta}\right)  }\left\Vert u_{0}^{n}-u_{0}\right\Vert \\
&  \leq c_{S}s^{-\beta}(t-s)^{\beta}\left\Vert u_{0}^{n}-u_{0}\right\Vert
,\ 0<s<t\leq T,
\end{align*}
implies that
\[
S(\cdot)u_{0}^{n}\rightarrow S(\cdot)u_{0}\text{ in }C_{\beta}^{\beta
}([0,T],V).
\]
Hence, $S(\cdot)\mathcal{K} $ is relatively compact in $C_{\beta}^{\beta
}([0,T],V).$

Joining the two results we obtain that $\mathcal{T}(B,\hat B,\mathcal{K} )$ is
relatively compact in $C_{\beta}^{\beta}([0,T],V)$.
\end{proof}

In the next result we address the existence of solutions to (\ref{Pm}).

\begin{theorem}
\label{t1} Under the above conditions on $A$, $F$ and $G$, given $T>0$, for
$\omega\in C^{\beta^{\prime}}([0,T],V) $ and $u_{0}\in V$ there exists at
least one mild pathwise solution $u\in C^{\beta}_{\beta}([0,T],V)$ to the
equation (\ref{P}) given by (\ref{Pm}).
\end{theorem}

\begin{proof}
We choose $\rho_{0}$ large enough such that $c_{T}\left\vert \!\left\vert
\!\left\vert \omega\right\vert \!\right\vert \!\right\vert _{\beta^{\prime}%
}K(\rho_{0})<\frac{1}{2}$. Therefore, from Lemma \ref{l6}, $\mathcal{T(}%
$\textperiodcentered$,\omega,u_{0}\mathcal{)}$ maps the ball
\[
B:=B(0,R)=\{u\in C_{\beta}^{\beta}([0,T],V):\Vert u\Vert_{\beta,\beta;\rho
_{0}}\leq R\},\quad\mbox{ with }\,R:=1+2c_{S}\Vert u_{0}\Vert
\]
into itself, that is, $\mathcal{T}(B,\omega,u_{0})\subset B.$

It is clear that $B$ is convex, bounded and closed and we know by Theorem
\ref{t3} that the operator $\mathcal{T}($\textperiodcentered$,\omega,u_{0})$
is compact. In order to apply Schauder's theorem, it only remains to check
that $\mathcal{T(}$\textperiodcentered$,\omega,u_{0}\mathcal{)}:B\mapsto B$ is
continuous. Assume that $(u^{n})_{n\in\mathbb{N}}\subset B$ is such that
$u^{n}(0)=u_{0}$ for every $n\in\mathbb{N}$, and $u^{n}\rightarrow u$ in
$C_{\beta}^{\beta}([0,T],V)$. We shall check that for $t\in\lbrack0,T],$
\[
\lim_{n\rightarrow\infty}\Vert\mathcal{T}(u^{n},\omega,u_{0})(t)-\mathcal{T}%
(u,\omega,u_{0})(t)\Vert=0.
\]

We only need to consider the integral terms. First of all, due to the
continuity of $F$,
\[
\left\Vert \int_{0}^{t}S(t-r)(F(u^{n}(r))-F(u(r)))dr\right\Vert \leq
c_{S,F}\int_{0}^{t}\|F(u^{n}(r))-F(u(r))\|dr
\]
and $f^{n}(r):=\|F(u^{n}(r))-F(u(r))\| \to0$ when $n\to\infty$, for
$r\in[0,t]$. Moreover, it has a trivial integrable majorant, given by
$2c_{F}+L_{F}(\|u^{n}\|_{\beta,\beta}+\|u\|_{\beta,\beta})\leq2c_{F}+2L_{F}
R.$

For the stochastic integral
\begin{align*}
&  \left\Vert \int_{0}^{t}S(t-r)(G(u^{n}(r))-G(u(r)))d\omega\right\Vert \\
&  \leq c\left\vert \!\left\vert \!\left\vert \omega\right\vert \!\right\vert
\!\right\vert _{\beta^{\prime},0,T}\int_{0}^{t}\left(  \frac{\left\Vert
S(t-r)\right\Vert _{L(V)}\left\Vert G(u^{n}(r))-G(u(r))\right\Vert _{L_{2}%
(V)}}{r^{\alpha}}\right. \\
&  +\int_{0}^{r}\frac{\left\Vert S(t-r)-S(t-q)\right\Vert _{L(V)}\left\Vert
G(u^{n}(r))-G(u(r))\right\Vert _{L_{2}(V)}}{(r-q)^{1+\alpha}}dq\\
&  +\int_{0}^{r}\left.  \frac{\left\Vert S(t-q)\right\Vert _{L(V)}\left\Vert
G(u^{n}(r))-G(u(r))-G(u^{n}(q))+G(u(q))\right\Vert _{L_{2}(V)}}%
{(r-q)^{1+\alpha}}dq\right)  (t-r)^{\alpha+\beta^{\prime}-1}dr\\
&  =A_{1}+A_{2}+A_{3}.
\end{align*}
Using the properties of $S$ and $G$ we have that the first two terms in the
last inequality can be estimated as follows:%
\begin{align*}
A_{1}+A_{2}  &  \leq c_{S}\left\vert \!\left\vert \!\left\vert \omega
\right\vert \!\right\vert \!\right\vert _{\beta^{\prime},0,T}\left(  \int%
_{0}^{t}\frac{L_{G}\left\Vert u^{n}(r)-u(r)\right\Vert }{r^{\alpha}%
}(t-r)^{\alpha+\beta^{\prime}-1}dr\right. \\
&  \left.  +\int_{0}^{t}\int_{0}^{r}\frac{L_{G}\left\Vert u^{n}%
(r)-u(r)\right\Vert (r-q)^{\beta}}{(t-r)^{\beta}(r-q)^{1+\alpha}%
}dq(t-r)^{\alpha+\beta^{\prime}-1}dr\right) \\
&  \leq c_{S,G} t^{\beta^{\prime}}\left\vert \!\left\vert \!\left\vert
\omega\right\vert \!\right\vert \!\right\vert _{\beta^{\prime},0,T} \left\Vert
u^{n}-u\right\Vert _{\beta,\beta}%
\end{align*}
thus we get that $A_{1}+A_{2}\rightarrow0$ as $n\rightarrow\infty.$

For the term $A_{3}$ we define the functions
\[
h^{n}(r,q)=\frac{\left\Vert G(u^{n}(r))-G(u(r))-G(u^{n}(q))+G(u(q))\right\Vert
_{L_{2}(V)}}{(r-q)^{1+\alpha}}(t-r)^{\alpha+\beta^{\prime}-1}.
\]
The Lipschitz property of $G$ implies that%
\[
h^{n}(r,q)\rightarrow0\text{ for a.a. }(r,q) \in D=\{(r,q):0\leq q\leq r\leq
t\}.
\]
On the other hand, we construct a majorant in the usual way:%
\begin{align*}
h^{n}(r,q)  &  \leq L_{G}\frac{\left\Vert u^{n}(r)-u^{n}(q)\right\Vert
+\left\Vert u(r)-u(q)\right\Vert }{(r-q)^{1+\alpha-\beta}(r-q)^{\beta}}%
\frac{q^{\beta}}{q^{\beta}}(t-r)^{\alpha+\beta^{\prime}-1}\\
&  \leq L_{G} (\left\Vert u^{n}\right\Vert _{\beta,\beta}+\left\Vert
u\right\Vert _{\beta,\beta})\frac{(t-r)^{\alpha+\beta^{\prime}-1}%
}{(r-q)^{1+\alpha-\beta}q^{\beta}}\\
&  \leq2R L_{G} \frac{(t-r)^{\alpha+\beta^{\prime}-1}}{(r-q)^{1+\alpha-\beta
}q^{\beta}}=f(r,q)\in L^{1}(D).
\end{align*}
By Lebesgue's theorem and Fubini's theorem we have that $A_{3}\rightarrow0$.

Finally, applying Schauder's fixed point theorem, the problem (\ref{P}) has at
least one mild pathwise solution given by (\ref{Pm}).
\end{proof}

\begin{remark}
As it can be easily seen in the proof of the Lemma \ref{l6}, the integrals are
well-defined just considering the norm of the space $C^{\beta}(0,T],V)$. As we
already pointed out in Section \ref{preli}, the factor $s^{\beta}$ that
appears in the norm of the space $C_{\beta}^{\beta}(0,T],V)$ is only required
due to the fact that the semigroup $S$ is not H\"{o}lder continuous at zero.
In fact, if the initial condition were in the space $V_{\beta},$ then we could
simply work in the space of H\"{o}lder continuous functions with exponent
$\beta$, since then
\[
\left\vert \!\left\vert \!\left\vert S(\cdot)u_{0}\right\vert \!\right\vert
\!\right\vert _{\beta,0,T}\leq\sup_{0\leq s<t\leq T}\frac{\Vert S(s)\Vert
_{L(V)}\Vert S(t-s)-\mathrm{Id}\Vert_{L(V_{\beta},V)}\Vert u_{0}%
\Vert_{V_{\beta}}}{(t-s)^{\beta}}\leq c_{S}\sup_{0\leq s<t\leq T}%
\frac{(t-s)^{\beta}\Vert u_{0}\Vert_{V_{\beta}}}{(t-s)^{\beta}}<\infty.
\]

Next we would like to establish, on the base of a concatenation procedure,
that every mild solution can be extended to be a \textit{globally defined}
mild solution, that is, it exists for any $t\geq0$. Denote by $u_{1}\in
C_{\beta}^{\beta}([0,T_{1}],V)$ the mild solution obtained in Theorem \ref{t1}
corresponding to the initial condition $u_{0}\in V$ and the driving path
$\omega\in C^{\beta^{\prime}}([0,T_{1}],V)$. Since by Corollary \ref{l1} we
know that $u_{1}(T_{1})\in V_{\beta}$, then considering as initial condition
$u_{2}(0):=u_{1}(T_{1})$ and taking the new driving path $\theta_{T_{1}}%
\omega\in C^{\beta^{\prime}}([0,T_{2}],V)$, thanks to Theorem \ref{t1} and the
above discussion we are able to obtain a mild solution $u_{2}\in C^{\beta
}([0,T_{2}],V)$.
\end{remark}

\begin{lemma}
\label{Concatenation}(Concatenation) Let $u_{1}\in C_{\beta}^{\beta}%
([0,T_{1}],V)$, $u_{2}\in C^{\beta}([0,T_{2}],V)$ be mild solutions to
(\ref{P}) with $u_{1}(0)\in V$, $u_{2}(0)=u_{1}(T_{1})$, for $\omega\in
C^{\beta^{\prime}}([0,T_{1}],V)$ and $\theta_{T_{1}}\omega\in C^{\beta
^{\prime}}([0,T_{2}],V)$, respectively, and let
\[
u(t)=\left\{
\begin{array}
[c]{c}%
u_{1}(t)\text{ if }t\in\lbrack0,T_{1}],\\
u_{2}(t-T_{1})\text{ if }t\in\lbrack T_{1},T_{1}+T_{2}].
\end{array}
\right.
\]
Then $u\in C_{\beta}^{\beta}([0,T_{1}+T_{2}],V)$ is a mild solution to
(\ref{P}) on $[0,T_{1}+T_{2}]$.
\end{lemma}

\begin{proof}
First, we need to show that $u\in C_{\beta}^{\beta}([0,T_{1}+T_{2}],V).$
Trivially
\[
\|u\|_{\infty,0,T_{1}+T_{2}} \leq\|u_{1}\|_{\infty,0,T_{1}}+\|u_{2}%
\|_{\infty,0,T_{2}}.
\]
On the other hand, since $u_{2}(0)=u_{1}(T_{1})$,
\begin{align*}
\left|  \! \left|  \! \left|  u\right|  \! \right|  \! \right|  _{\beta
,\beta,0,T_{1}+T_{2}}  &  \leq\max\bigg( \sup_{0< s<t\leq T_{1}}s^{\beta}%
\frac{\|u_{1}(t)-u_{1}(s)\|}{(t-s)^{\beta}},\sup_{T_{1}\leq s<t\leq
T_{1}+T_{2}}s^{\beta}\frac{\|u_{2}(t-T_{1})-u_{2}(s-T_{1})\|}{(t-s)^{\beta}%
},\\
&  \qquad\qquad\sup_{0< s< T_{1} <t\leq T_{1}+T_{2}} s^{\beta}\frac
{\|u_{2}(t-T_{1})-u_{1}(s)\|}{(t-s)^{\beta}} \bigg)\\
&  \leq\max\bigg(\left|  \! \left|  \! \left|  u_{1}\right|  \! \right|  \!
\right|  _{\beta,\beta,0,T_{1}},(T_{1}+T_{2})^{\beta}\left|  \! \left|  \!
\left|  u_{2}\right|  \! \right|  \! \right|  _{\beta,0,T_{2}}, T_{1}^{\beta
}\sup_{T_{1} <t\leq T_{1}+T_{2}} \frac{\|u_{2}(t-T_{1})-u_{2}(T_{1}-T_{1}%
)\|}{(t-T_{1})^{\beta}}\\
&  \qquad\qquad+\sup_{0< s< T_{1} } s^{\beta}\frac{\|u_{1}(T_{1})-u_{1}%
(s)\|}{(T_{1}-s)^{\beta}} \bigg)\\
&  \leq\max(\left|  \! \left|  \! \left|  u_{1}\right|  \! \right|  \!
\right|  _{\beta,\beta,0,T_{1}},(T_{1}+T_{2})^{\beta} \left|  \! \left|  \!
\left|  u_{2}\right|  \! \right|  \! \right|  _{\beta,0,T_{2}}, T_{1}^{\beta
}\left|  \! \left|  \! \left|  u_{2}\right|  \! \right|  \! \right|
_{\beta,0,T_{2}}+\left|  \! \left|  \! \left|  u_{1}\right|  \! \right|  \!
\right|  _{\beta,\beta,0,T_{1}})\\
&  =\left|  \! \left|  \! \left|  u_{1}\right|  \! \right|  \! \right|
_{\beta,\beta,0,T_{1}}+(T_{1}+T_{2})^{\beta}\left|  \! \left|  \! \left|
u_{2}\right|  \! \right|  \! \right|  _{\beta,0,T_{2}}<\infty.
\end{align*}

Second, we will prove that $u$ satisfies the integral equality (\ref{Pm}).
Since this is obvious if $t\in\lbrack0,T_{1}]$, let $t\in(T_{1},T_{1}+T_{2}]$.
Thus, in virtue of properties (ii) and (iii) of Lemma \ref{l3}, $u(t)=u_{2}%
\left(  t-T_{1}\right)  $ is such that
\begin{align*}
u(t)  &  =S(t-T_{1})u_{1}(T_{1})+\int_{0}^{t-T_{1}}S(t-T_{1}-r)F(u_{2}%
(r))dr+\int_{0}^{t-T_{1}}S(t-T_{1}-r)G(u_{2}(r))d\theta_{T_{1}}\omega\\
&  =S(t-T_{1})\left(  S(T_{1})u_{1}(0)+\int_{0}^{T_{1}}S(T_{1}-r)F(u_{1}%
(r))dr+\int_{0}^{T_{1}}S(T_{1}-r)G(u_{1}(r))d\omega\right) \\
&  \qquad+\int_{0}^{t-T_{1}}S(t-T_{1}-r)F(u_{2}(r))dr+\int_{0}^{t-T_{1}%
}S(t-T_{1}-r)G(u_{2}(r))d\theta_{T_{1}}\omega\\
&  =S(t)u_{1}(0)+\int_{0}^{T_{1}}S(t-r)F(u_{1}(r))dr+\int_{0}^{T_{1}%
}S(t-r)G(u_{1}(r))d\omega\\
&  \qquad+\int_{T_{1}}^{t}S(t-r)F(u_{2}(r-T_{1}))dr+\int_{T_{1}}%
^{t}S(t-r)G(u_{2}(r-T_{1}))d\omega\\
&  =S(t)u(0)+\int_{0}^{t}S(t-r)F(u(r))dr+\int_{0}^{t}S(t-r)G(u(r))d\omega.
\end{align*}

\end{proof}

The above method can be repeated in such a way that the corresponding
solutions to (\ref{P}) are globally defined. \newline

We also prove that the solutions satisfy the translation property.

\begin{lemma}
\label{Translation}(Translation) Let $u(\cdot)$ be a mild solution to
(\ref{P}) on $[0,T]$ with $u_{0}\in V$ for $\omega\in C^{\beta^{\prime}%
}([0,T],V)$ and let $0<s<T$. Then the function $v(\cdot)=u(\cdot+s)$ is a mild
solution on $[0,T-s]$ with $v(0)=u(s)$ for the driving path $\theta_{s}\omega$.
\end{lemma}

For the proof of this lemma we can apply the techniques from the last lemma.

\section{Multivalued non-autonomous and random dynamical systems}

\label{MDS}

We start this section by introducing the general concept of multivalued
non-autonomous and random dynamical systems. Later we will apply it to the set
of solutions of the problem (\ref{P}).

Let $\Omega$ be some set. On $\Omega$ we define a flow of non-autonomous
perturbations $\theta:\mathbb{R} \times\Omega\mapsto\Omega$ by
\[
\theta_{0} \omega=\omega,\quad\theta_{t} \circ\theta_{\tau}=\theta_{t+\tau},
\quad t,\tau\in\mathbb{R} ,
\]
for $\omega\in\Omega$.

We now give the definition of a metric dynamical system, that is a general
model for a noise. On a probability space $(\Omega,\mathcal{F},\mathbb{P})$ we
consider a $\mathcal{B}(\mathbb{R})\otimes\mathcal{F},\mathcal{F}$-measurable
flow $\theta$ such that $\theta_{t}\mathbb{P}=\mathbb{P}$ for every
$t\in\mathbb{R}$. Often it is also assumed that $\mathbb{P}$ is ergodic with
respect to the flow $\theta$, which means that, in addition to the invariance
property of $\mathbb{P}$ defined above, given an invariant set $A\in
\mathcal{F}$ (that is, $\theta_{t}A=A$, for all $t\in\mathbb{R}$), we have
either $\mathbb{P}(A)=0$ or $\mathbb{P}(A)=1$. Then the quadruple
$(\Omega,\mathcal{F},\mathbb{P},\theta)$ is called an ergodic metric dynamical system.

\begin{remark}
For the following results we only need a \textit{semiflow} instead of a flow
$\theta$, that is, defined on $\mathbb{R} ^{+}$. However, for further
considerations regarding the existence of random attractors in a forthcoming
paper, we will need to deal with a flow as introduced above. For the existence
of a flow defined on $\mathbb{R} $ we refer to \cite{Sinai} Page 240.
\end{remark}

Denote by $P_{f}(V)$ the set of all non-empty closed subsets of $V.$

\begin{definition}
\label{MNDS}Consider a flow of non-autonomous perturbations $\theta:\mathbb{R}
\times\Omega\mapsto\Omega$. A multivalued mapping $\Phi:\mathbb{R}^{+}%
\times\Omega\times V\rightarrow P_{f}(V)$ is called a multivalued
non--autonomous dynamical system (MNDS) if:

\begin{itemize}
\item[i)] $\Phi(0,\omega,\cdot)=\mathrm{id}_{V}$,

\item[ii)] $\Phi(t+\tau,\omega,x)\subset\Phi(t,\theta_{\tau}\omega,\Phi
(\tau,\omega,x))$ \textit{\ (cocycle property) \ for all} \textit{\ }%
$t,\tau\in\mathbb{R}^{+},x\in V,\omega\in\Omega.$
\end{itemize}

It is called a strict MNDS if $\Phi(t+\tau,\omega,x)=\Phi(t,\theta_{\tau
}\omega,\Phi(\tau,\omega,x))$ \textit{\ for all} \textit{\ }$t,\tau
\in\mathbb{R}^{+},\ x\in V,\ \omega\in\Omega.$

Assume now that $(\Omega, \mathcal{F} ,\mathbb{P} , \theta)$ is an (ergodic)
metric dynamical system. An MNDS is called a multivalued random dynamical
system (MRDS) if the multivalued mapping $(t,\omega,x)\rightarrow\Phi
(t,\omega,x)$ is $\mathcal{B}(\mathbb{R}^{+})\otimes\mathcal{F}\otimes
\mathcal{B}(V)$ measurable, i.e.
\[
\{(t,\omega,x):\Phi(t,\omega,x)\cap\mathcal{O} \not =\emptyset\}\in
\mathcal{B}(\mathbb{R}^{+})\otimes\mathcal{F}\otimes\mathcal{B}(V)
\]
for every open set $\mathcal{O} $ of $V$.
\end{definition}

A suitable concept of continuity in the setting of multivalued dynamical
systems is the following one.

\begin{definition}
$\Phi(t, \omega, \cdot)$ is called upper semicontinuous at $x_{0}$ if for
every open neighborhood $\mathcal{O} \subset V$ of the set $\Phi(t, \omega,
x_{0})$ there exists $\delta>0$ such that if $\|x_{0}-y\| <\delta$ then
$\Phi(t, \omega, y)\in\mathcal{O} $. $\Phi(t, \omega, \acute{a})$ is called
upper semicontinuous if it is upper semicontinuous at every $x_{0}$ in $V$.
\end{definition}

This definition can be extended to the one of upper semicontinuity with
respect to all variables assuming that $\Omega$ is a Polish space. We are now
able to formulate a general condition ensuring that an MNDS defines an MRDS.
For the proof, see Lemma 2.5 in \cite{CGSV08}.

\begin{lemma}
\label{sep} Let $\Omega$ be a Polish space and let $\mathcal{F}$ be the
associated Borel $\sigma$-algebra. Suppose that $(t, \omega, x) \mapsto
\Phi(t,\omega,x)$ is upper semicontinuous. Then $\Phi$ is measurable in the
sense of Definition \ref{MNDS}.
\end{lemma}

As a result, an MNDS $\Phi$ that it is upper semicontinuous with respect to
its three variables becomes an MRDS, provided that the system of
non-autonomous perturbations is a metric dynamical system. But we stress once
again that above we have requested the separability of $\Omega$.\newline

In what follows we present two examples of metric dynamical systems describing
a Gau\ss \, noise given by the fractional Brownian motion with any Hurst
parameter $H\in(0,1)$ (fBm to short). Given $H\in(0,1)$, a continuous centered
Gau{\ss }ian process $\beta^{H}(t)$, $t\in\mathbb{R}$, with the covariance
function
\[
\mathbb{E}\beta^{H}(t)\beta^{H}(s)=\frac{1}{2}(|t|^{2H}+|s|^{2H}%
-|t-s|^{2H}),\qquad t,\,s\in\mathbb{R}%
\]
is called a two--sided one-dimensional fractional Brownian motion, and $H$ is
the Hurst parameter. Assume that $Q$ is a bounded and symmetric positive
linear operator on $V$ which is of trace class, i.e., for a complete
orthonormal basis $\{e_{i}\}_{i\in{\mathbb{N}}}$ in $V$ there exists a
sequence of nonnegative numbers $\{q_{i}\}_{i\in{\mathbb{N}}}$ such that $tr
Q:=\sum_{i=1}^{\infty}q_{i} <\infty$. Then a continuous $V$-valued fractional
Brownian motion $B^{H}$ with covariance operator $Q$ and Hurst parameter $H$
is defined by%

\[
B^{H}(t)=\sum_{i=1}^{\infty} \sqrt{q_{i}}e_{i} \beta_{i}^{H}(t),\quad
t\in\mathbb{R},
\]
where $\{\beta_{i}^{H}(t)\}_{i\in{\mathbb{N}}}$ is a sequence of
stochastically independent one-dimensional fBm. \newline

In virtue of Kolmogorov's theorem we know that $B^{H}$ has a continuous
version, see \cite{Bauer} Theorem 39.3. Hence we can consider the canonical
interpretation of an fBm: let $C_{0}:=C_{0}(\mathbb{R},V)$ be the space of
continuous functions on $\mathbb{R}$ with values in $V$. Here and below the
subindex means that these functions are zero at zero, equipped with the
compact open topology. Let $\mathcal{F}=\mathcal{B} (C_{0}(\mathbb{R} ,V))$ be
the associated Borel-$\sigma$-algebra, ${\mathbb{P}}$ the distribution of the
fBm $B^{H}$ and $\{\theta_{t}\}_{t\in\mathbb{R}}$ the flow of Wiener shifts
given by (\ref{shift}). In that way, $(C_{0}(\mathbb{R} ,V),\mathcal{B}
(C_{0}(\mathbb{R} ,V)),\mathbb{P} ,\theta)$ is an ergodic metric dynamical
system, see \cite{MasSchm04} and \cite{GSch11}. The foundation of that
property can be found in \cite{Bauer} Theorem 38.6.

Furthermore, this (canonical) process has a version $\omega\in C^{\gamma}%
_{0}:=C^{\gamma}_{0}(\mathbb{R} ,V)$, that is, $\omega(0)=0$ and it is
$\gamma$-H{\"o}lder continuous on any interval $[-n,n]$ for all $\gamma<H$,
see \cite{Bauer}, Theorem 39.4. This regularity of the fractional Brownian
motion makes this process to be the main example fitting our abstract setting.

However, the space $C^{\gamma}_{0}$ is not suitable for our further purposes
since it is not separable, and we need a Polish space (see Lemma \ref{sep}
above). Nevertheless, in order to give a meaning to the stochastic integrals
we still need to consider H\"older continuous functions and the continuous
dependence of the integral with respect to the integrators in some subspace of
H\"older continuous functions. In fact, in what follows we consider a second
example consisting of a subspace of H\"older continuous functions that it is
separable.\newline

For $n\in\mathbb{N} $ and a general parameter $\gamma\in(0,1)$ we define
\[
C^{0,\gamma}([-n,n],V):=\overline{C^{\infty}([-n,n],V)}^{C^{\gamma}([-n,n],V)}%
\]
which is a closed linear subspace of $C^{\gamma}([-n,n],V)$ and in particular
a separable Banach sepace, see Friz and Victoir \cite{FrizVictoir},
Proposition 5.36. Let us denote by $C_{0}^{0,\gamma}:=C_{0}^{0,\gamma
}(\mathbb{R} ,V)$ the Fr\'echet space given by the subset of $C_{0}$ whose
elements restricted to $[-n,n]$ are in $C^{0,\gamma}([-n,n],V)$, so it is also
a separable complete metric space. The generating norms are the $\gamma
$-H{\"o}lder norms of functions on $[-n,n]$. According to the Wiener's
characterization, see \cite{FrizVictoir}, Theorem 5.31, the space
$C^{0,\gamma}([-n,n],V)$ can be also expressed as
\begin{align}
\label{Wc}%
\begin{split}
C^{0,\gamma}([-n,n],V)  &  =\bigg\{ x\in C^{\gamma}([-n,n],V): \lim
_{\delta\rightarrow0} \sup_{\substack{-n\leq s<t\leq n, \\t-s<\delta}}
\frac{\|x(t) -x(s)\|}{|t-s|^{\gamma}}=0\bigg\}\\
&  =\bigg\{ x\in C^{\gamma}([-n,n],V): \lim_{\delta\rightarrow0}
\sup_{\substack{-n\leq s<t\leq n, \\t-s<\delta}} \left|  \! \left|  \! \left|
x \right|  \! \right|  \! \right|  _{ \gamma,s,t}=0\bigg\}.
\end{split}
\end{align}

Next we define a new ergodic metric dynamical system modeling the fractional
Brownian motion where the set $\Omega=C_{0}^{0,\gamma}$ is separable. Needless
to say that this result has its own interest, since to our knowledge the path
space of the metric dynamical system modeling the fBm considered in the
literature has been given so far either by $C_{0}$ or $C_{0}^{\gamma}$.
\medskip

Consider the ergodic metric dynamical system $(C_{0},\mathcal{B}%
(C_{0}),\mathbb{P},\theta)$ where $\theta$ is the Wiener shift on $C_{0}$ and
$\mathbb{P}$ is the distribution of the fBm. We would like to prove that
defining $\mathbb{P}^{\prime}(A)=\mathbb{P}(B),\,A=B\cap C_{0}^{0,\gamma
},\,B\in\mathcal{B}(C_{0})$ and $\theta^{\prime}$ being the restriction of
$\theta$ to $\mathbb{R}\times\mathcal{B}(C_{0}^{0,\gamma})$, the quadruple
$(C_{0}^{0,\gamma},\mathcal{B}(C_{0}^{0,\gamma}),\mathbb{P}^{\prime}%
,\theta^{\prime})$ is also an ergodic metric dynamical system.\newline

We start establishing the connections between the two Borel $\sigma$-algebras
$\mathcal{B} (C_{0}^{0,\gamma})$ and $\mathcal{B} (C_{0})$.

\begin{lemma}
\label{le1} We have $\mathcal{B} (C_{0}^{0,\gamma})=\mathcal{B} (C_{0})\cap
C_{0}^{0,\gamma}$.
\end{lemma}

\begin{proof}
Since $C_{0}^{0,\gamma}$ is continuously embedded in $C_{0}$, the inclusion
$\mathcal{B} (C_{0})\cap C_{0}^{0,\gamma}\subset\mathcal{B}(C_{0}^{0,\gamma})$
follows from Vishik and Fursikov \cite{VF} Theorem II.2.1, so we need to check
that $\mathcal{B}(C_{0}^{0,\gamma})\subset\mathcal{B} (C_{0})\cap
C_{0}^{0,\gamma}.$

A countable generator of $\mathcal{B}(C_{0}^{0,\gamma})$ consists of all open
sets of the form%
\[
V(\omega_{0};n,\varepsilon_{1},\varepsilon_{2})=\{\omega\in C_{0}^{0,\gamma
}:\left\Vert \omega-\omega_{0}\right\Vert _{\infty,-n,n}<\varepsilon
_{1},\ \left\vert \!\left\vert \!\left\vert \omega-\omega_{0}\right\vert
\!\right\vert \!\right\vert _{\gamma,-n,n}<\varepsilon_{2}\},
\]
where $\omega_{0}$ belongs to a countable set of $C_{0}^{0,\gamma}$,
$n\in\mathbb{N}$, $\varepsilon_{i}\in\mathbb{Q} ^{+}\setminus\{0\}$.

For any $\delta>0$ and $\varepsilon_{i}\in\mathbb{Q} ^{+}\setminus\{0\}$ we
define the sets%
\[
A_{\delta}(\omega_{0};n,\varepsilon_{1},\varepsilon_{2})=\left\{  \omega\in
C_{0}^{0,\gamma}:\left\Vert \omega-\omega_{0}\right\Vert _{\infty
,-n,n}<\varepsilon_{1},\sup_{\substack{s,t\in\lbrack-n,n], \\t-s\geq\delta
}}\frac{\| \omega(t)-\omega_{0}(t)-\omega(s)+\omega_{0}(s)\|}{(t-s)^{\gamma}%
}<\varepsilon_{2}\right\}  .
\]
It is straightforward to see that%
\[
A_{\delta_{1}}(\omega_{0};n,\varepsilon_{1},\varepsilon_{2})\subset
A_{\delta_{2}}(\omega_{0};n,\varepsilon_{1},\varepsilon_{2})\text{ if }%
\delta_{1}<\delta_{2}
\]
and%
\begin{equation}
V(\omega_{0};n,\varepsilon_{1},\varepsilon_{2})\subset A_{\delta}(\omega
_{0};n,\varepsilon_{1},\varepsilon_{2})\text{, }\forall\delta>0.
\label{VInclusion}%
\end{equation}
We will prove that%
\[
V(\omega_{0};n,\varepsilon_{1},\varepsilon_{2})=\bigcap_{k\in\mathbb{N}%
}A_{\frac{1}{k}}(\omega_{0};n,\varepsilon_{1},\varepsilon_{2}).
\]
In view of (\ref{VInclusion}) it is enough to check that%
\[
\bigcap_{k\in\mathbb{N}}A_{\frac{1}{k}}(\omega_{0};n,\varepsilon
_{1},\varepsilon_{2})\subset V(\omega_{0};n,\varepsilon_{1},\varepsilon_{2}).
\]
Let $\omega\in\bigcap_{k\in\mathbb{N}}A_{\frac{1}{k}}(\omega_{0}%
;n,\varepsilon_{1},\varepsilon_{2})$ be arbitrary. Since $\omega\in
C_{0}^{0,\gamma}$, by the Wiener's characterization (\ref{Wc}), there exists
$\delta(\varepsilon_{2})$ such that
\[
\sup_{\substack{s,t\in\lbrack-n,n], \\0<t-s<\delta(\varepsilon_{2})}}\frac{\|
\omega(t)-\omega_{0}(t)-\omega(s)+\omega_{0}(s)\|}{(t-s)^{\gamma}}%
<\varepsilon_{2}.
\]
In particular $\omega\in A_{\delta(\varepsilon_{2})}$, so we also have that%
\[
\sup_{\substack{s,t\in\lbrack-n,n], \\t-s\geq\delta(\varepsilon_{2})}}\frac{\|
\omega(t)-\omega_{0}(t)-\omega(s)+\omega_{0}(s)\|}{(t-s)^{\gamma}}%
<\varepsilon_{2}.
\]
Thus, $\omega\in V(\omega_{0};n,\varepsilon_{1},\varepsilon_{2}).$

For $\delta>0$ fixed, we observe that in the space $C_{0}([-n,n],V)$ the norm
$\left\Vert \text{\textperiodcentered}\right\Vert _{\delta,-n,n}$ given by%
\[
\left\Vert \omega\right\Vert _{\delta,-n,n}=\left\Vert \omega\right\Vert
_{\infty,-n,n}+\sup_{\substack{s,t\in\lbrack-n,n], \\t-s\geq\delta}%
}\frac{\|\omega(t)-\omega(s)\|}{(t-s)^{\gamma}}
\]
is equivalent to the standard norm $\left\Vert \omega\right\Vert
_{\infty,-n,n}$. Indeed,%
\[
\left\Vert \omega\right\Vert _{\infty,-n,n}\leq\left\Vert \omega\right\Vert
_{\delta,-n,n}\leq\left\Vert \omega\right\Vert _{\infty,-n,n}+2\frac
{\left\Vert \omega\right\Vert _{\infty,-n,n}}{\delta^{\gamma}}\leq\left(
1+\frac{2}{\delta^{\gamma}}\right)  \left\Vert \omega\right\Vert
_{\infty,-n,n}.
\]

This implies that the set
\[
\widetilde{A}_{\delta}(\omega_{0};n,\varepsilon_{1},\varepsilon_{2})=\left\{
\omega\in C_{0}:\left\Vert \omega-\omega_{0}\right\Vert _{\infty
,-n,n}<\varepsilon_{1},\sup_{\substack{s,t\in\lbrack-n,n], \\t-s\geq\delta
}}\frac{\|\omega(t)-\omega_{0}(t)-\omega(s)+\omega_{0}(s)\|}{(t-s)^{\gamma}%
}<\varepsilon_{2}\right\}
\]
is open in $C_{0}$. Hence,
\[
A_{\delta}(\omega_{0};n,\varepsilon_{1},\varepsilon_{2})=\widetilde{A}%
_{\delta}(\omega_{0};n,\varepsilon_{1},\varepsilon_{2})\cap C_{0}^{0,\gamma
}\in\mathcal{B}(C_{0})\cap C_{0}^{0,\gamma},
\]
so%
\[
V(\omega_{0};n,\varepsilon_{1},\varepsilon_{2})=\bigcap_{k\in\mathbb{N}%
}A_{\frac{1}{k}}(\omega_{0};n,\varepsilon_{1},\varepsilon_{2})\in
\mathcal{B}(C_{0})\cap C_{0}^{0,\gamma}.
\]

\end{proof}

\begin{lemma}
\label{l22} We have that $C_{0}^{0,\gamma}\in\mathcal{B} (C_{0})$.
\end{lemma}

\begin{proof}
First of all, for $a<b$, for $\omega\in C_{0}$ we define the following
mapping:
\[
f_{a,b}(\omega)=\left\{
\begin{array}
[c]{lcl}%
\left\vert \!\left\vert \!\left\vert \omega\right\vert \!\right\vert
\!\right\vert _{\gamma,a,b} & : & \mathrm{if}\ \omega\in C^{\gamma
}([a,b],V),\\
\infty & : & \mathrm{if}\ \omega\in C([a,b],V)\backslash C^{\gamma}([a,b],V).
\end{array}
\right.
\]

Then the mapping $f_{a,b}:(C_{0},\mathcal{B}(C_{0}))\rightarrow(\bar
{\mathbb{R}}^{+},\mathcal{B}(\bar{\mathbb{R}}^{+}))$ is measurable. In order
to prove this statement, we consider
\[
f_{a,b,k}(\omega)=\sup_{\substack{a\leq s<t\leq b,\\t-s\geq\frac{1}{k}}%
}\frac{\Vert\omega(t)-\omega(s)\Vert}{|t-s|^{\gamma}},
\]
which is a continuous mapping on $C_{0}$ with values in $\mathbb{R}^{+}$. If
we prove that the following property
\begin{equation}
\lim_{k\rightarrow\infty}f_{a,b,k}(\omega)=f_{a,b}(\omega). \label{eq1b}%
\end{equation}
holds true, then the measurability of $f_{a,b}$ follows, since the pointwise
limit of measurable functions is also a measurable function.

Note that the sequence $(f_{a,b,k})_{k\in\mathbb{N} }$ is non decreasing, so
that there exists its limit in $\bar{\mathbb{R}}^{+}$ and this limit is
smaller than or equal to $f_{a,b}(\omega)$ for every $\omega\in C_{0}$. On the
other hand, by the definition of supremum, there exists a sequence $(s_{n},
t_{n})_{n\in\mathbb{N} }$ with $a\le s_{n}<t_{n}\le b$ such that
\[
\lim_{n\to\infty}\frac{\|\omega(t_{n})-\omega(s_{n})\|}{|t_{n}-s_{n}|^{\gamma
}}=f_{a,b}(\omega).
\]

We can select an increasing subsequence $(k_{n^{\prime}})$ such that
$1/{k_{n^{\prime}}}\le t_{n}-s_{n}$ and hence
\[
\frac{\|\omega(t_{n})-\omega(s_{n})\|}{|t_{n}-s_{n}|^{\gamma}}\le
f_{a,b,k_{n^{\prime}}}(\omega)\le f_{a,b}(\omega)
\]
and the left hand side converges to $f_{a,b}(\omega)$, which shows \eqref{eq1b}.

Now, for $\omega\in C_{0}$ consider the mapping
\[
g_{k,n}(\omega)=\sup_{\substack{-n\leq s<t\leq n,\\s,\,t\in\mathbb{Q}%
,\\t-s<\frac{1}{k}}}f_{s,t}(\omega).
\]
Then the mapping $g_{k,n}$ is $(C_{0},\mathcal{B}(C_{0})),(\bar{\mathbb{R}%
}^{+},\mathcal{B}(\bar{\mathbb{R}}^{+}))$-measurable, which follows since this
supremum is taken over countably many measurable elements. In addition, we can
prove that
\begin{equation}
g_{k,n}(\omega)=\sup_{\substack{-n\leq s<t\leq n,\\t-s<\frac{1}{k}}%
}f_{s,t}(\omega). \label{gm}%
\end{equation}
Straightforwardly, the right hand side of (\ref{gm}) is larger than or equal
to the left hand side. Conversely, for fixed $n$, let $(s_{m}^{n},t_{m}^{n})$
with $-n\leq s_{m}^{n}<t_{m}^{n}\leq n$ and $t_{m}^{n}-s_{m}^{n}<1/k$ such
that
\[
\lim_{m\rightarrow\infty}f_{s_{m}^{n},t_{m}^{n}}(\omega)=\sup
_{\substack{-n\leq s<t\leq n,\\t-s<\frac{1}{k}}}f_{s,t}(\omega).
\]
Then we find $\bar{s}_{m}^{n},\,\bar{t}_{m}^{n}\in\mathbb{Q}$ such that
$-n\leq\bar{s}_{m}^{n}\leq s_{m}^{n}<t_{m}^{n}\leq\bar{t}_{m}^{n}\leq n$,
$\bar{t}_{m}^{n}-\bar{s}_{m}^{n}<1/k$. Hence
\[
f_{s_{m}^{n},t_{m}^{n}}(\omega)\leq f_{\bar{s}_{m}^{n},\bar{t}_{m}^{n}}%
(\omega)
\]
which gives the opposite inequality in (\ref{gm}). Note that if $s_{m}^{n}=-n$
we can set $\bar{s}_{m}^{n}=-n\in\mathbb{Q}$ and similarly for $t_{m}^{n}=n$.
We finally define the mapping
\[
h_{n}(\omega)=\limsup_{k\rightarrow\infty}g_{k,n}(\omega)\in\bar{\mathbb{R}%
}^{+}%
\]
which is measurable in $\mathcal{B}(C_{0})$. We stress that if $h_{n}%
(\omega)=0$ then we indeed have
\[
0=\lim_{k\rightarrow\infty}g_{k,n}(\omega)\in\bar{\mathbb{R}}^{+}.
\]
Hence
\[
A_{n}:=h_{n}^{-1}(\{0\})\in\mathcal{B}(C_{0})
\]
and then, by the Wiener's characterization and (\ref{gm}), we finally obtain
\[
A_{n}=\{\omega\in C_{0}:\omega|_{[-n,n]}\in C^{0,\gamma}([-n,n],V)\}
\]
that implies
\[
C_{0}^{0,\gamma}=\bigcap_{n}A_{n}\in\mathcal{B}(C_{0}).
\]

\end{proof}

Before proving the main result of this section, let us remind here several
properties that are necessary.
%\begin{definition}
%The symmetric difference of two sets $A$ and $B$ is given by
%$$A\Delta B:= (A\cap B^c)\cup(A^c\cap B)=(A\cup B)\cap(A\cap B)^c$$
%where, as usual, $A^c$ denotes the complementary of $A$.
%\end{definition}

\begin{definition}
\label{der} Given a metric dynamical system $(\Omega, \mathcal{F}%
,\mathbb{P},\theta)$, $A$ is invariant mod $\mathbb{P}$ if $\mathbb{P}%
(A\Delta\theta_{t} A)=0,$ for every $t\in\mathbb{R} $.
%And $A$ and $B$ are equal mod $\mathbb P$ if $\mathbb P(A\Delta B)=0.$

\end{definition}

The following result can be found in Walters \cite{W}, Theorem 1.5:

\begin{lemma}
\label{ler} The metric dynamical system $(\Omega, \mathcal{F},\mathbb{P}%
,\theta)$ is ergodic if and only if every $\theta$-invariant set mod
$\mathbb{P}$ has measure zero or one.
\end{lemma}

Now we can prove the main theorem of this section.

\begin{theorem}
\label{mt} The quadruple $(C_{0}^{0,\gamma},\mathcal{B} (C_{0}^{0,\gamma
}),\mathbb{P} ^{\prime},\theta^{\prime})$ is an ergodic metric dynamical
system where
\[
\mathbb{P} ^{\prime}(A)=\mathbb{P} (B),\,A=B\cap C_{0}^{0,\gamma},
\,B\in\mathcal{B} (C_{0})
\]
and $\theta^{\prime}$ is the restriction of $\theta$ to $\mathbb{R}
\times\mathcal{B} (C_{0}^{0,\gamma})$.
\end{theorem}

\begin{proof}
Notice that $\mathbb{P}(C_{0}^{0,\gamma})=1$, which follows from the property
that the fBm has paths in $C_{0}^{0,\gamma}$. In fact, from \cite{Bauer}
Theorem 39.4 and \cite{Kunita90} Theorem 1.4.1, we know that $\omega\in
C_{0}^{\gamma^{\prime}}$ for any $\gamma<\gamma^{\prime}<H$. Hence, for every
$n\in\mathbb{N}$ in particular $\omega\in C^{\gamma}([-n,n],V)$ and
\begin{align*}
\lim_{\delta\rightarrow0}\sup_{\substack{-n\leq s<t\leq n,\\t-s<\delta}%
}\frac{\Vert\omega(t)-\omega(s)\Vert}{|t-s|^{\gamma}}  &  =\lim_{\delta
\rightarrow0}\sup_{\substack{-n\leq s<t\leq n,\\t-s<\delta}}\frac{\Vert
\omega(t)-\omega(s)\Vert}{|t-s|^{\gamma^{\prime}}}|t-s|^{\gamma^{\prime
}-\gamma}\\
&  \leq\lim_{\delta\rightarrow0}\sup_{\substack{-n\leq s<t\leq n}}\frac
{\Vert\omega(t)-\omega(s)\Vert}{|t-s|^{\gamma^{\prime}}}\delta^{\gamma
^{\prime}-\gamma}\\
&  \leq\left\vert \!\left\vert \!\left\vert \omega\right\vert \!\right\vert
\!\right\vert _{\gamma^{\prime},-n,n}\lim_{\delta\rightarrow0}\delta
^{\gamma^{\prime}-\gamma}=0.
\end{align*}

Furthermore, by Lemma \ref{le1} and Lemma \ref{l22}, $\mathbb{P}^{\prime}$ is
defined on $\mathcal{B}(C_{0}^{0,\gamma})$. To check that $\mathbb{P}^{\prime
}$ is well--defined, let $A\in\mathcal{B}(C_{0}^{0,\gamma})$ be such that for
$B_{1},\,B_{2}\in\mathcal{B}(C_{0})$ we have
\[
A=B_{i}\cap C_{0}^{0,\gamma},\quad i=1,\,2.
\]
It is easy to check that
\[
(B_{1}\cap C_{0}^{0,\gamma})\Delta(B_{2}\cap C_{0}^{0,\gamma})=C_{0}%
^{0,\gamma}\cap(B_{1}\Delta B_{2}),
\]
thus, because the symmetric difference of a set with itself is the empty set
and $\mathbb{P}(C_{0}^{0,\gamma})=1$, we have
\[
0=\mathbb{P}((B_{1}\cap C_{0}^{0,\gamma})\Delta(B_{2}\cap C_{0}^{0,\gamma
}))=\mathbb{P}(B_{1}\Delta B_{2}),
\]
so that $\mathbb{P}(B_{1})=\mathbb{P}(B_{2})$ and therefore $\mathbb{P}%
^{\prime}(A)=\mathbb{P}(B_{i})$. This property, together with the $\sigma
$-additivity and the fact that trivially $\mathbb{P}^{\prime}(C_{0}^{0,\gamma
})=1$, implies that $\mathbb{P}^{\prime}$ is a probability measure.\newline

We now prove that $\theta_{t}^{\prime}$ has values in $C_{0}^{0,\gamma}$ for
$t\in\mathbb{R}$. Suppose that $\omega\in C_{0}^{0,\gamma}$. Then for any
$n\in\mathbb{N}$ there exists a sequence $(\omega_{m}^{n})_{m\in\mathbb{N}}$
converging to $\omega$ in $C^{\gamma}([-n,n],V)$ where $\omega_{m}^{n}\in
C^{\infty}([-n,n],V),\,\omega_{m}^{n}(0)=0$. For some $t\in\mathbb{R}$ we
consider the sequence $(\omega_{m}^{[t]+1+n}),$ which converges to $\omega$ on
$C^{\gamma}([-[t]-1-n,[t]+1+n],V)$, where for $t\geq0$ the value $[t]$ is the
largest integer less or equal than $t$ and for $t<0$ the value $[t]$ is the
smallest integer larger or equal than $t$. Then $(\theta_{t}\omega
_{m}^{[t]+1+n}|_{[-n,n]})_{m\in\mathbb{N}}$ converges to $\theta_{t}\omega$ in
$C^{\gamma}([-n,n],V)$, so that $\theta_{t}\omega\in C_{0}^{0,\gamma}$, and
hence $\theta_{t}^{\prime}C_{0}^{0,\gamma}\subset C_{0}^{0,\gamma}$, for
$t\in\mathbb{R}$. Since $\theta_{t}$ is a bijection with inverse $\theta_{-t}%
$, we also obtain that
\begin{equation}
\theta_{t}C_{0}^{0,\gamma}=\theta_{t}^{\prime}C_{0}^{0,\gamma}=C_{0}%
^{0,\gamma}\quad\text{for all }t\in\mathbb{R}.\label{bi}%
\end{equation}
The flow property of $\theta^{\prime}$ follows easily form the flow property
of $\theta$.

We prove now that $\theta^{\prime}$ is $\mathcal{B} (\mathbb{R} )\otimes
\mathcal{B} (C_{0}^{0,\gamma}),\mathcal{B} (C_{0}^{0,\gamma})$ measurable.
Note that for $A\in\mathcal{B} (C_{0}^{0,\gamma})$, as a consequence of
(\ref{bi}),
\begin{equation}
\label{eqeq}(\theta^{\prime})^{-1}(A) = (\theta^{\prime})^{-1}(B\cap
C_{0}^{0,\gamma})= \theta^{-1}(B)\cap\theta^{-1}(C_{0}^{0,\gamma}%
)\in(\mathcal{B} (\mathbb{R} )\otimes\mathcal{B} (C_{0})) \cap(\mathbb{R}
\times C_{0}^{0,\gamma} ).
\end{equation}
Indeed the second equality follows by
\begin{align*}
\theta^{-1}(B\cap C_{0}^{0,\gamma})  &  =\{(t,\omega)\in\mathbb{R} \times
C_{0}:\theta_{t}\omega\in B\cap C_{0}^{0,\gamma}\} .
\end{align*}
but when $\theta_{t}\omega\in C_{0}^{0,\gamma}$ so $\omega\in C_{0}^{0,\gamma
}$. Applying Vishik and Fursikov \cite{VF} Theorem II.2.1 to \eqref{eqeq}, we
obtain that
\[
(\theta^{\prime})^{-1}(A) \in\mathcal{B} (\mathbb{R} \times C_{0}^{0,\gamma
})=\mathcal{B} (\mathbb{R} )\otimes\mathcal{B} (C_{0}^{0,\gamma}).
\]
Note that the last equality above follows by the separability of $\mathbb{R} $
and $C^{0,\gamma}_{0}$, see \cite{Bauer}, Chapter 35, Remark 1.

Next we prove the invariance of the measure $\mathbb{P} ^{\prime}$, which is
derived from the invariance of the measure $\mathbb{P} $ and the following
chain of equalities: for all $A\in\mathcal{B} (C_{0}^{0,\gamma})$,
\begin{align*}
\mathbb{P} ^{\prime}(\theta_{t}^{\prime-1}A)=  &  \mathbb{P} ^{\prime}%
(\theta_{t}^{\prime-1}(B\cap C_{0}^{0,\gamma}))=\mathbb{P} ^{\prime}%
(\theta_{t}^{-1}(B\cap C_{0}^{0,\gamma}))=\mathbb{P} ^{\prime}(\theta_{t}%
^{-1}B\cap C_{0}^{0,\gamma})=\mathbb{P} (\theta_{t}^{-1}B)\\
=  &  \mathbb{P} (B)=\mathbb{P} ^{\prime}(A).
\end{align*}

Finally, we prove the ergodicity of $\mathbb{P}^{\prime}$. Let $A\in
\mathcal{B}(C_{0}^{0,\gamma})$ be a $\theta^{\prime}$ invariant set. Then for
any $B\in\mathcal{B}(C_{0})$ such that $A=B\cap C_{0}^{0,\gamma}$, since
$\mathbb{P}(C_{0}^{0,\gamma})=1$,
\begin{align*}
\mathbb{P}((\theta_{t}^{-1}B)\Delta B)=  &  \mathbb{P}^{\prime}(C_{0}%
^{0,\gamma}\cap((\theta_{t}^{-1}B)\Delta B))=\mathbb{P}^{\prime}((\theta
_{t}^{-1}(B\cap C_{0}^{0,\gamma}))\Delta(B\cap C_{0}^{0,\gamma}))\\
=  &  \mathbb{P}^{\prime}((\theta_{t}^{\prime-1}(B\cap C_{0}^{0,\gamma
}))\Delta(B\cap C_{0}^{0,\gamma}))=\mathbb{P}^{\prime}((\theta_{t}^{\prime
-1}A)\Delta A)=\mathbb{P}(\emptyset)=0,
\end{align*}
By Definition \ref{der} we conclude that $B$ is $\theta$ invariant mod
$\mathbb{P}$, hence the ergodicity of $\mathbb{P}$ implies that either
$\mathbb{P}(B)=0$ or 1, see Lemma \ref{ler}. As a result, taking once more
into account that $\mathbb{P}(C_{0}^{0,\gamma})=1$ we conclude that either
$\mathbb{P}^{\prime}(A)=0$ or 1, and therefore $\mathbb{P}^{\prime}$ is ergodic.
\end{proof}

\section{Multivalued non-autonomous and random dynamical systems for
(\ref{P})}

\label{mnds}

In this section we deal with the multivalued dynamical system related to
problem (\ref{P}). First of all, we emphasize that Lemma \ref{Concatenation}
implies in particular that every mild pathwise solution can be extended to a
globally defined one, that is, it exists for any $t\geq0$.\newline

Denote by $\mathcal{F}(u_{0},\omega)$ the set of all globally defined mild
solutions with initial value $u_{0}\in V$ for $\omega\in C_{0}^{0,\beta
^{\prime}}$ and consider the ergodic metric dynamical system $(C_{0}%
^{0,\beta^{\prime}},\mathcal{B} (C_{0}^{0,\beta^{\prime}}),\mathbb{P}
^{\prime},\theta^{\prime})$ introduced in Theorem \ref{mt} for $\gamma
=\beta^{\prime}$.

Denoting by $P\left(  V\right)  $ the set of all non-empty subsets of $V$, we
define the (possibly) multivalued operator $\Phi:\mathbb{R}^{+}\times
\Omega\times V\rightarrow P\left(  V\right)  $ by%
\[
\Phi(t,\omega,u_{0})=\{u(t):u\in\mathcal{F}(u_{0},\omega)\}.
\]

\begin{lemma}
\label{Cocycle}$\Phi(t+s,\omega,u_{0})=\Phi(t,\theta_{s}\omega,\Phi
(s,\omega,u_{0}))$, for all $t,s\geq0,\ \omega\in\Omega$, $u_{0}\in V$. That
is, $\Phi$ is a strict MRDS.
\end{lemma}

The proof follows easily from Lemmas \ref{Concatenation} and \ref{Translation}%
.\newline

Next we would like to prove that the norm of any solution $u \in
\mathcal{F}(u_{0},\omega)$ is uniformly bounded in any interval $[0,T]$.

\begin{lemma}
\label{BoundSolutions2}Let us consider the balls $B_{V}(0,R)$ and
$B_{C_{0}^{0,\beta^{\prime}}}(0,\hat R)$. There exists $C(R,\hat R,T)$ such
that for any $u_{0} \in B_{V}(0,R)$, $\omega\in B_{C_{0}^{0,\beta^{\prime}}%
}(0,\hat R)$ and $u\in\mathcal{F} (u_{0},\omega)$ one has%
\[
\left\Vert u\right\Vert _{\beta,\beta,0,T}\leq C(R,\hat R,T).
\]

\end{lemma}

\begin{proof}
For $u\in\mathcal{F} (u_{0},\omega)$, according to (\ref{eq20}) we have that
\begin{equation}
\label{eq20b}\|u\|_{\beta,\beta;\rho}\le c_{S}\|u_{0}\|+c_{T} \left|  \!
\left|  \! \left|  \omega\right|  \! \right|  \! \right|  _{\beta^{\prime}%
}K(\rho)(1+\|u\|_{\beta,\beta;\rho}),
\end{equation}
where $\lim_{\rho\to\infty}K(\rho)=0$ and $c_{T}$ is a constant that also
depends on the constants related to $F$, $G$ and $S$. Therefore, we can choose
$\rho$ large enough such that
\[
c_{T}\left|  \! \left|  \! \left|  \omega\right|  \! \right|  \! \right|
_{\beta^{\prime}}K(\rho)\leq c_{T}\hat R K(\rho) <\frac{1}{2},
\]
and thus, for all $\omega\in B_{C_{0}^{0,\beta^{\prime}}}(0,\hat R)$, we have
\[
K(\rho)<\frac{1}{2c_{T}\hat R }.
\]
Plugging this information in the estimate (\ref{eq20b}) we have
\[
\|u\|_{\beta,\beta;\rho, 0,T}(1-c_{T} \hat R K(\rho))\leq(c_{S} \|u_{0}%
\|+c_{T} \hat R K(\rho))
\]
and therefore
\[
\|u\|_{\beta,\beta;\rho, 0,T}\leq2c_{S} R+1.
\]
Since $\rho=\rho(T,\hat R)$ and $\|u\|_{\beta,\beta;\rho, 0,T}$ is equivalent
to $\|u\|_{\beta,\beta, 0,T}$, we obtain the result.
\end{proof}

\begin{theorem}
\label{ConvergSolutions}Let $u_{0}^{n}\rightarrow u_{0}$ in $V$ and $\omega\in
C_{0}^{0,\beta^{\prime}}$ be fixed. Then every sequence $(u^{n})_{n\in
\mathbb{N} } \in\mathcal{F}(u_{0}^{n},\omega)$ possesses a subsequence
$u^{n_{k}}$ such that
\[
u^{n_{k}}|_{[0,T]}\rightarrow u|_{[0,T]} \text{ in }C_{\beta}^{\beta
}([0,T],V),
\]
for all $T>0$, where $u\in\mathcal{F}(u_{0},\omega)$.
\end{theorem}

\begin{proof}
First we fix $T>0$. It follows from Lemma \ref{BoundSolutions2} that $u^{n}$
is bounded in $C_{\beta}^{\beta}([0,T],V)$, hence by Theorem \ref{t3} the
sequence of mappings
\[
\lbrack0,T]\ni t\mapsto\int_{0}^{t}S(t-r)F(u^{n}(r))dr+\int_{0}^{t}%
S(t-r)G(u^{n}(r))d\omega
\]
is relatively compact in $C_{\beta}^{\beta}([0,T],V)$. On the other hand, the
inequality%
\[
\left\Vert (S(t)-S\left(  s\right)  )(u_{0}^{n}-u_{0})\right\Vert \leq
c_{S}s^{-\beta}(t-s)^{\beta}\left\Vert u_{0}^{n}-u_{0}\right\Vert ,\ 0<s<t\leq
T,
\]
implies that
\begin{equation}
S(\cdot)u_{0}^{n}\rightarrow S(\cdot)u_{0}\text{ in }C_{\beta}^{\beta
}([0,T],V), \label{ConvergTerm1}%
\end{equation}
and therefore we have that $u^{n}($\textperiodcentered$)$ is relatively
compact in $C_{\beta}^{\beta}([0,T],V)$. We conclude that up to a subsequence
\[
u^{n}\rightarrow u\text{ in }C_{\beta}^{\beta}([0,T],V).
\]
By a diagonal argument the result is true for an arbitrary $T>0$.

It remains to check that $u\in\mathcal{F}(u_{0},\omega)$. In view of
\[
\left\Vert S(t)u_{0}^{n}-S(t)u_{0}\right\Vert \rightarrow0,
\]
it suffices to verify that for any $t\in\lbrack0,T]$
\[
\int_{0}^{t}S(t-r)F(u^{n}(r))dr+\int_{0}^{t}S(t-r)G(u^{n}(r))d\omega
\rightarrow\int_{0}^{t}S(t-r)F(u(r))dr+\int_{0}^{t}S(t-r)G(u(r))d\omega,
\]
which has a similar proof as the continuity of $\mathcal{T}$ in Theorem
\ref{t1}.
\end{proof}

\begin{corollary}
\label{CompactValues}The map $\Phi$ has compact values, and thus $\Phi$ is a
strict MNDS.
\end{corollary}

As a consequence of the previous results, we can also establish the following
property, which will be crucial when looking at the existence of attractors
for (\ref{P}).

\begin{corollary}
\label{USC}The map $u_{0}\mapsto\Phi(t,\omega,u_{0})$ is upper semicontinuous.
\end{corollary}

\begin{proof}
If this is not true, for some $t,\omega,u_{0}$ there exists a neighborhood
$\mathcal{U} $ of $\Phi(t,\omega,u_{0})$ in $V$ and sequences $y^{n}\in
\Phi(t,\omega,u_{0}^{n})$, $u_{0}^{n}\rightarrow u_{0}$, such that
$y^{n}\not \in \mathcal{U} $. But $y^{n}=u^{n}\left(  t\right)  $ with
$u^{n}\in\mathcal{F}(u_{0}^{n},\omega)$, so by Theorem \ref{t3} \ we have that
up to a subsequence $y_{n}\rightarrow y\in\Phi(t,\omega,u_{0})$, which is a
contradiction with $y_{n}\not \in \mathcal{U} $.
\end{proof}

We are now ready to prove upper semicontinuity with respect to all variables.

\begin{theorem}
\label{ConvergSolutions2}Let $u_{0}^{n}\rightarrow u_{0}$ in $V$ and
$\omega^{n}\rightarrow\omega$. Then every sequence $u^{n}\in\mathcal{F}%
(u_{0}^{n},\omega^{n})$ possesses a subsequence $u^{n_{k}}$ such that%
\[
u^{n_{k}}\rightarrow u\in\mathcal{F}(u_{0},\omega)\text{ in }C_{\beta}^{\beta
}([0,T],V),
\]
where $T>0$ is arbitrary.
\end{theorem}

\begin{proof}
In view of Lemma \ref{BoundSolutions2} the sequence $(u^{n})_{n\in\mathbb{N}}$
is bounded in $C_{\beta}^{\beta}([0,T],V)$ for any $T>0$. Hence, Theorem
\ref{t3}\ implies the existence of $u$ such that up to a subsequence%
\[
u^{n}\rightarrow u\text{ in }C_{\beta}^{\beta}([0,T],V)\text{ for any }T>0.
\]
It remains to prove that $u\in\mathcal{F}(u_{0},\omega)$. This is equivalent
to checking that $u=\mathcal{T}(u,\omega,u_{0})$, which will follow from%
\[
\mathcal{T}(u^{n},\omega^{n},u_{0}^{n})(t)\rightarrow\mathcal{T}%
(u,\omega,u_{0})(t)\text{ in }V\text{ for any }t\in\lbrack0,T].
\]
As it is clear that $S(t)u_{0}^{n}\rightarrow S(t)u_{0}$ in $V$, we only need
to consider the integral terms. For the deterministic integral we can follow
the same steps as in Theorem \ref{t1}. For the stochastic integral, we split
the difference in the following way:%
\begin{align*}
&  \left\Vert \int_{0}^{t}S(t-r)G(u^{n}(r))d\omega^{n}-\int_{0}^{t}%
S(t-r)G(u(r))d\omega\right\Vert \\
&  \leq\left\Vert \int_{0}^{t}S(t-r)(G(u^{n}(r))-G(u(r)))d\omega
^{n}\right\Vert \\
&  +\left\Vert \int_{0}^{t}S(t-r)G(u(r))d(\omega^{n}-\omega)\right\Vert
=:I_{1}+I_{2}.
\end{align*}

Arguing as in the proof of Theorem \ref{t1} we obtain that $I_{1}\rightarrow0$.

For $I_{2}$ we deduce that
\begin{align*}
&  \left\Vert \int_{0}^{t}S(t-r)G(u(r))d(\omega^{n}-\omega)\right\Vert \\
&  \leq c\left\vert \!\left\vert \!\left\vert \omega^{n}-\omega\right\vert
\!\right\vert \!\right\vert _{\beta^{\prime},0,T}\int_{0}^{t}\left(
\frac{\left\Vert S(t-r)\right\Vert _{L(V)}\left\Vert G(u(r))\right\Vert
_{L_{2}(V)}}{r^{\alpha}}\right. \\
&  +\int_{0}^{r}\frac{\left\Vert S(t-r)-S(t-q)\right\Vert _{L(V)}\left\Vert
G(u(r))\right\Vert _{L_{2}(V)}}{(r-q)^{1+\alpha}}dq\\
&  \left.  +\int_{0}^{r}\frac{\left\Vert S(t-q)\right\Vert _{L(V)}\left\Vert
G(u(r))-G(u(q))\right\Vert _{L_{2}(V)}}{(r-q)^{1+\alpha}}dq\right)
(t-r)^{\alpha+\beta^{\prime}-1}dr=:A_{1}+A_{2}+A_{3}.
\end{align*}

The first term is estimated by%
\[
A_{1}\leq c_{S,G} \left\vert \!\left\vert \!\left\vert \omega^{n}%
-\omega\right\vert \!\right\vert \!\right\vert _{\beta^{\prime},0,T}\left(
1+\|u\|_{\infty,0,T}\right)  \int_{0}^{t}r^{-\alpha}(t-r)^{\beta^{\prime
}+\alpha-1}dr,
\]
so $A_{1}\rightarrow0$.

Using (\ref{eq30}) the second term is estimated by%
\begin{align*}
A_{2}  &  \leq c_{S,G}\left\vert \!\left\vert \!\left\vert \omega^{n}%
-\omega\right\vert \!\right\vert \!\right\vert _{\beta^{\prime},0,T}\left(
1+\|u\|_{\infty,0,T}\right)  \int_{0}^{t}\int_{0}^{r}\frac{\left(  r-q\right)
^{\beta}}{\left(  t-r\right)  ^{\beta}(r-q)^{1+\alpha}}dq(t-r)^{\alpha
+\beta^{\prime}-1}dr\\
&  =c_{S,G}\left\vert \!\left\vert \!\left\vert \omega^{n}-\omega\right\vert
\!\right\vert \!\right\vert _{\beta^{\prime},0,T}\left(  1+\|u\|_{\infty
,0,T}\right)  \int_{0}^{t}r^{\beta-\alpha}(t-r)^{\alpha+\beta^{\prime}%
-\beta-1}dr.
\end{align*}
Hence, $A_{2}\rightarrow0.$

For the last term we have%
\begin{align*}
A_{3}  &  \leq c_{S,G}\left\vert \!\left\vert \!\left\vert \omega^{n}%
-\omega\right\vert \!\right\vert \!\right\vert _{\beta^{\prime},0,T}\int%
_{0}^{t}\int_{0}^{r}\frac{q^{\beta}\left\Vert u(r)-u(q)\right\Vert }{q^{\beta
}\left(  r-q\right)  ^{\beta}(r-q)^{1+\alpha-\beta}}dq(t-r)^{\alpha
+\beta^{\prime}-1}dr\\
&  \leq c_{S,G} \left\vert \!\left\vert \!\left\vert \omega^{n}-\omega
\right\vert \!\right\vert \!\right\vert _{\beta^{\prime},0,T}\left\Vert
u\right\Vert _{\beta,\beta,0,T}\int_{0}^{t} r^{-\alpha}(t-r)^{\alpha
+\beta^{\prime}-1}dr.
\end{align*}
Thus, $A_{3}\rightarrow0.$

The proof is now complete.
\end{proof}

Notice that the convergence of $u^{n}$ to $u$ in $C^{\beta}_{\beta}([0,T],V)$
implies the uniform convergence of $u^{n}(t)$ to $u(t)$ in $V$. Then, a direct
consequence of Theorem \ref{ConvergSolutions2} is the following result.

\begin{corollary}
\label{ConvergSequences}If $u_{0}^{n}\rightarrow u_{0}$ in $V$, $\omega
^{n}\rightarrow\omega_{0}$ in $C^{\beta^{\prime}}([0,T],V)$, $t^{n}\rightarrow
t_{0}$ in $\mathbb{R} ^{+}$ and $y^{n}\in\Phi(t^{n},\omega^{n},u_{0}^{n})$,
then there exists a subsequence $y^{n_{k}}$ such that $y^{n_{k}}\rightarrow
y_{0}\in\Phi(t_{0},\omega_{0},u_{0}).$
\end{corollary}

We finally can establish the main result of this section.

\begin{theorem}
The mapping $(t,\omega,x)\rightarrow\Phi(t,\omega,x)$ is $\mathcal{B}%
(\mathbb{R}^{+})\otimes\mathcal{F}\otimes\mathcal{B}(V)$ measurable. Hence,
$\Phi$ is a MRDS, where $\mathcal{F=B(}C_{0}^{0,\beta^{\prime}}\mathcal{)}$.
\end{theorem}

\begin{proof}
From Corollary \ref{CompactValues} we already know that $\Phi$ is a strict
MNDS. On the other hand, from Corollary \ref{ConvergSequences} it follows that
the map $(t,\omega,x)\rightarrow\Phi(t,\omega,x)$ is upper semicontinuous in
the multivalued sense, and, since $C_{0}^{0,\beta^{\prime}}$ is separable,
this property implies the measurability of $\Phi$, see Lemma \ref{sep}. In
other words, $\Phi$ is a strict MRDS.
\end{proof}

%%%%%%%%%%%%%EXAMPLE%%%%%%%%%%
Finally, we give an example of a parabolic partial differential equation whose
set of solutions generates a multivalued random dynamical system.

\begin{example}
\textrm{Let $D\subset\mathbb{R} ^{d}$ be a bounded domain with regular enough
boundary. Consider the space $V=L^{2}(D)$ with usual norm denoted by
$\|.\|_{V}$ and a complete orthonormal base given by $(e_{i})_{i\in\mathbb{N}
}$. Assume that $A$ is given by the Laplacian on $D$ with homogenous Dirichlet
boundary condition. The operator $-A$ with domain $D(-A)=H^{2}(D)\cap
H_{0}^{1}(D)$ is a strictly positive and symmetric operator with a compact
inverse, generating an analytic semigroup $S$ in $V$.\newline}

\textrm{We consider $f:\mathbb{R}  \mapsto\mathbb{R} $ to be a continuous
mapping with at most linear growth. We define the nonlinear drift $F:V \mapsto
V$ as the corresponding Nemytskii operator given by
\[
F(u)[x]=f(u(x)), \quad\text{for }u\in V,\, x\in D.
\]
}

\textrm{Now we introduce the diffusion term. In order to do that, let $g: D
\times D \times\mathbb{R}  \mapsto\mathbb{R} $ be a Lipschitz continuous
function in the following sense:
\[
|g(x,y,z_{1})-g(x,y,z_{2})|\leq L(x) |z_{1}-z_{2}|, \quad x,y\in D,\;
z_{1},z_{2}\in\mathbb{R} ,
\]
where $L\in V$. Now we define
\[
G(u)(v)[x]=\int_{D} g(x,y,u(y))v(y)dy,\quad\text{for }u,\,v\in V.
\]
We can see that $G$ is well defined as a mapping $G:V\mapsto L_{2}(V)$, that
is, with values in the Hilbert-Schmidt operators from $V$ into $V$. In fact,
for $u\in V$,
\begin{align*}
\|G(u)\|^{2}_{L_{2}(V)} & =\sum_{j} \|G(u)(e_{j})\|^{2}_{V}=\sum_{j} \int_{D}
|G(u)(e_{j}) [x]|^{2} dx=\sum_{j} \int_{D} \bigg(\int_{D} g(x,y,u(y))e_{j}%
(y)dy \bigg)^{2} dx\\
& =\int_{D} \sum_{j} \bigg(\int_{D} g(x,y,u(y))e_{j}(y)dy \bigg)^{2}
dx\leq\int_{D} \|g(x,\cdot,u(\cdot))\|_{V}^{2} dx <\infty,
\end{align*}
where above we have applied Parseval's inequality. Furthermore, $G$ is
Lipschitz continuous, since for $u_{1}, u_{2}\in V$, in a similar way as
before we obtain
\begin{align*}
\|G(u_{1})-G(u_{2})\|^{2}_{L_{2}(V)} & =\sum_{j} \int_{D} \bigg(\int_{D}
(g(x,y,u_{1}(y))-g(x,y,u_{2}(y)))e_{j}(y)dy \bigg)^{2} dx\\
& \leq\int_{D} \|g(x,\cdot,u_{1}(\cdot))-g(x,\cdot,u_{2}(\cdot))\|_{V}^{2}
dx\\
& \leq\bigg(\int_{D} L^{2}(x)dx \bigg) \|u_{1}-u_{2}\|_{V}^{2}=\|L\|_{V}^{2}
\|u_{1}-u_{2}\|_{V}^{2}.
\end{align*}
}
\end{example}

\section*{Acknowledgement}

M.J. Garrido-Atienza was partially supported by FEDER and Spanish Ministerio
de Econom\'{\i}a y Competitividad, project MTM2015-63723-P and by Junta de
Andaluc\'{\i}a under Proyecto de Excelencia. J. Valero was partially supported
by FEDER and Spanish Ministerio de Econom\'{\i}a y Competitividad, projects
MTM2015-63723-P and MTM2016-74921-P.


\begin{thebibliography}{99}                                                                                               %


\bibitem {Arn98}L. Arnold. \newblock {\em Random dynamical systems}.
\newblock Springer Monographs in Mathematics. Springer-Verlag, Berlin, 1998.

\bibitem {Bauer}H. Bauer. \textit{Probability theory,} de Gruyter Studies in
Mathematics, 23. Walter de Gruyter, Berlin, 1996.

\bibitem {CGSV08}T. Caraballo, M.J. Garrido-Atienza, B. Schmalfu\ss \, and J.
Valero, Non-autonomous and random attractors for delay random semilinear
equations without uniqueness, \textit{Discrete Contin. Dyn. Syst.}, 21 (2008), 415--443.

\bibitem {CGSchV10}T. Caraballo, M.J. Garrido-Atienza, B.~Schmalfu{\ss }\, and
J. Valero, Asymptotic behavior of a stochastic semilinear dissipative
functional equation without uniqueness of solutions,
\newblock {\em Discrete and continuous dynamical systems, series B}, 14(2)
(2010), 439--455.

\bibitem {CHSchV}T. Caraballo, X. Han, Xiaoying, B. Schmalfu\ss and J. Valero,
Random attractors for stochastic lattice dynamical systems with infinite
multiplicative white noise. \textit{Nonlinear Anal.}, 130 (2016), 255--278.

\bibitem {CLV}T. Caraballo, J. A. Langa and J. Valero, Global attractors for
multivalued random dynamical systems, \textit{Nonlinear Anal.}, {48} (2002), 805--829.

\bibitem {CGGSch14}Y. Chen, H. Gao, M.~J. Garrido-Atienza and B. Schmalfu\ss ,
Pathwise solutions of SPDEs driven by H\"{o}lder-continuous integrators with
exponent larger than 1/2 and random dynamical systems, \emph{Discrete Contin.
Dyn. Syst.,} 34(1) (2014), 79--98.

\bibitem {Sinai}I.P. Cornfeld, S.V. Fomin and Y. G. Sinai, \textit{Ergodic
theory}. Springer-Verlag, New York, 1982.

\bibitem {DGNSch17}L.H. Duc, M.J. Garrido-Atienza, A. Neuenkirch and B.
Schmalfu\ss , Exponential stability of stochastic evolution equations driven
by small fractional Brownian motion with Hurst parameter in $(1/2,1)$,
\textit{Journal of Differential Equations,} 264(2) (2018), 1119--1145.

\bibitem {FrizVictoir}P.K.Friz and N.B. Victoir, \textit{Multidimensional
stochastic processes as rough paths}, Cambridge University Press, Cambridge, 2010.

\bibitem {GGSch}M.J. Garrido-Atienza, H. Gao and B. Schmalfu\ss , Random
attractor for stochastic evolution equations driven by fractional Brownian
motion, \textit{SIAM J. Math. Anal. } 46(4) (2014), 2281--2309.

\bibitem {GLS10}M.J. Garrido-Atienza, K. Lu and B. Schmalfu\ss , Random
dynamical systems for stochastic partial differential equations driven by a
fractional Brownian motion. \textit{Discrete Contin. Dyn. Syst. Ser. B}, 14(2)
(2010), 473--493.

\bibitem {GMSch}M.J. Garrido-Atienza, B. Maslowski and B. Schmalfu\ss , Random
attractors for stochastic equations driven by a fractional Brownian motion.
\textit{Internat. J. Bifur. Chaos Appl. Sci. Engrg.} {20}(9) (2010), 2761--2782.

\bibitem {GNSch16}M.~J. Garrido-Atienza, A. Neuenkirch and B.~Schmalfu\ss ,
{Asymptotical stability of differential equations driven by H{\"o}%
lder--continuous paths}, \textit{Journal of Dynamics and Differential
Equations}, 30(1) (2018), 359--377.

\bibitem {GSch11}M.J. Garrido-Atienza and B. Schmalfu\ss , Ergodicity of the
infinite dimensional fractional Brownian motion. {\textit{J}. Dynam.
Differential Equations}, 23(3) (2011), 671--681.

\bibitem {G}{B. Gess}, \emph{Random Attractors for Stochastic Porous Media
Equations perturbed by space-time linear multiplicative noise}, C.R. Acad.
Sci. Paris, Ser. I, 350, (2012), pp.~299--302.

\bibitem {GLR}{B. Gess, W. Liu and M. R\"ockner}, {Random attractors for a
class of stochastic partial differential equations driven by general additive
noise}, \textit{Journal of Differential Equations}, 251(4--5) (2011), 1225--1253.

\bibitem {Gu}{A. Gu}, {Random attractors of stochastic lattice dynamical
systems driven by fractional Brownian motion}, \textit{Int. J. Bifurcation
Chaos}, 23, 1350041 (2013) [9 pages] DOI: 10.1142/S0218127413500417

\bibitem {GrAnh}W. Grecksch and V. V. Anh, A parabolic stochastic differential
equation with fractional Brownian motion input, \textit{Statist. Probab.
Lett.}, 41: 337--345, 1999.

\bibitem {GuLeTin}M. Gubinelli, A. Lejay and S. Tindel, Young integrals and
SPDEs, \emph{Potential Anal. An Inter-national Journal Devoted to the
Interactions between Potential Theory, Probability Theory, Geometry and
Functional Analysis}, 25: 307--326, 2006.

\bibitem {Holte}J.M. Holte. Discrete Gronwall Lemma and Applications.

\bibitem {Kunita90}H.~Kunita.
\newblock {\em Stochastic flows and stochastic differential equations}.
\newblock Cambridge University Press, 1990.

\bibitem {lunardi}A. Lunardi,
\newblock {\em Analytic semigroups and optimal regularity in parabolic
problems}. \newblock Progress in Nonlinear Differential Equations and their
Applications, 16. Birkh\"{a}user Verlag, Basel, 1995.

\bibitem {MasNua03}B. Maslowski and D. Nualart. \newblock Evolution equations
driven by a fractional {B}rownian motion. \newblock {\it J. Funct. Anal.},
202(1) (2003), 277--305.

\bibitem {MasSchm04}B. Maslowski and B. Schmalfu{\ss }, \newblock Random
dynamical systems and stationary solutions of differential equations driven by
the fractional {B}rownian motion. \newblock {\it Stochastic Anal. Appl.},
22(6) (2004), 1577--1607.

\bibitem {NuaRas02}D. Nualart and A. R{\u{a}}{\c{s}}canu.
\newblock Differential equations driven by fractional {B}rownian motion.
\newblock {\it Collect. Math.}, 53(1) (2002), 55--81.

\bibitem {Samko}S.G. Samko, A.A. Kilbas and O.I. Marichev,
\newblock {\em Fractional integrals and derivatives: theory and applications}.
\newblock Gordon and Breach Science Publishers (Switzerland and Philadelphia,
Pa., USA), 1993.

\bibitem {TinTuVi}S. Tindel, C. Tudor and F. Viens, Stochastic evolution
equations with fractional Brownian motion, \textit{Probability Theory and
Related Fields}, 127 (2003), 186--204.

\bibitem {VF}M.I. Vishik and A.V. Fursikov, \textit{Mathematical problems of
statistical hydromechanics}, Kluwer Academic Publishers, Dordrecht, 1988.

\bibitem {W}P. Walters, \textit{An Introduction to Ergodic Theory,} Volume 79
of Graduate Texts in Mathematics. Springer, New York, 1982.

\bibitem {You36}L.C. Young, An integration of H{\"{o}}der type, connected with
Stieltjes integration, \newblock {\it Acta Math.}, 67 (1936), 251--282.

\bibitem {Zah98}M.~Z{\"{a}}hle. \newblock Integration with respect to fractal
functions and stochastic calculus. {I}.
\newblock {\it Probab. Theory Related Fields}, 111(3) (1998), 333--374.
\end{thebibliography}
\end{document}